\documentclass[12pt]{amsart}
\usepackage{amssymb,mathrsfs}

\newtheorem{theorem}{Theorem}[section]
\newtheorem{lemma}{Lemma}[section]
\newtheorem{corollary}{Corollary}[section]

\theoremstyle{definition}

\theoremstyle{remark}
\newtheorem{rk}{Remark}[section]

\numberwithin{equation}{section}

\newcommand {\SN} {{\mathbb N}}
\newcommand {\SR} {{\mathbb R}}

\newcommand {\SZ} {{\mathbb Z}}

\newcommand{\be}{\begin{equation}}
\newcommand{\ee}{\end{equation}}
\newcommand{\bea}{\begin{eqnarray}}
\newcommand{\eea}{\end{eqnarray}}
\newcommand{\bv}\boldsymbol{}

\newcommand{\hs}{H(x,y,z;P_s)}
\newcommand{\hsd}{H(x,y,z;P_s)-H(x-\Delta,y,z;P_s)}

\begin{document}

\title[Divisors of shifted primes]{Divisors of shifted primes}

\author{Dimitris Koukoulopoulos}
\address{Department of  Mathematics\\
University of Illinois at Urbana-Champaign\\
1409 West Green Street\\
Urbana\\ IL 61801\\ U.S.A.} \email{{\tt dkoukou2@math.uiuc.edu}}


\date{\today}

\maketitle

\begin{abstract} We bound from below the number of shifted primes $p+s\le x$
that have a divisor in a given interval $(y,z]$. Kevin Ford has
obtained upper bounds of the expected order of magnitude on this
quantity as well as lower bounds in a special case of the parameters
$y$ and $z$. We supply here the corresponding lower bounds in a
broad range of the parameters $y$ and $z$. As expected, these bounds
depend heavily on our knowledge about primes in arithmetic
progressions. As an application of these bounds, we determine the
number of shifted primes that appear in a multiplication table up to
multiplicative constants.
\end{abstract}

\section{Introduction}\label{intro} When one studies the multiplicative
structure of the integers a natural and important question that
arises is how many integers possess a divisor in a prescribed
interval $(y,z]$. More precisely, for $y<z$ and $x\ge1$ define
$$H(x,y,z)=\lvert\{n\le x:\exists d|n\;\text{with}\;y<d\le
z\}\rvert.$$ The study of this function was initiated by Besicovitch
\cite{bes}\;and was further developed by Erd\H os \cite{erd3},
\cite{erd4}, \cite{erd2}\;and Tenenbaum \cite{ten1}, \cite{ten2},
who obtained bounds on $H(x,y,z)$ in various cases of the parameters
$y$ and $z$. In his seminal paper \cite{ten}\;Tenenbaum focused on
estimating $H(x,y,z)$ for all $x,y,z$ and he obtained reasonably
sharp bounds on it. A consequence of Tenenbaum's work was the
realization that, for fixed $x$ and $y$, as $z$ varies in $(y,x]$
the behavior of $H(x,y,z)$ changes when $z$ is around $y+y(\log
y)^{-\log4+1}$, $2y$ and $y^2$. The problem of establishing the
correct order of magnitude of $H(x,y,z)$ was completely resolved by
Ford in his profound work\;\cite{kf}, where he discovered a striking
connection between the distribution of the prime factors of integers
with a divisor in $(y,z]$ and random walks with certain constraints.
We state here the core of the main theorem in \cite{kf}. First, for
a given pair $(y,z)$ with $2\le y<z$ define $\eta,u,\beta$ and $\xi$
by \be\label{def1}z=e^\eta y=y^{1+u},\quad\eta=(\log
y)^{-\beta},\quad\beta=\log4-1+\frac\xi{\sqrt{\log\log y}}.\ee
Furthermore, put
$$z_0(y)=y\exp\{(\log y)^{-\log 4+1}\}\approx y+y(\log y)^{-\log4+1},$$
$$G(\beta)=\begin{cases}\displaystyle\frac{1+\beta}{\log2}\log\Bigl(\frac{1+\beta}{e\log 2}\Bigr)+1&0\le\beta\le\log4-1,\cr
\beta&\log4-1\le\beta,\end{cases}$$ and
$$\delta=1-\frac{1+\log\log2}{\log2}=0.086071\dots$$ Lastly, here and for the rest of this paper
the notation $f\asymp g$ means that $f\ll g$ and $g\ll f$.
Constants implied by $\ll$, $\gg$ and $\asymp$ are absolute unless
otherwise specified, e.g. by a subscript.

\renewcommand{\labelenumi}{(\alph{enumi})}

\begin{theorem}[Ford \cite{kf}]\label{thm1} Let $x>100000$ and
$100\le y\le z-1$ with $z\le x$.
\begin{enumerate}\item  If $y\le\sqrt{x}$, then
$$\frac{H(x,y,z)}x\asymp\begin{cases}\log(z/y)=\eta&y+1\le z\le z_0(y),\cr
\displaystyle\frac{\beta}{\max\{1,-\xi\}(\log
y)^{G(\beta)}}&z_0(y)\le z\le 2y,\cr \displaystyle
u^\delta\Bigl(\log\frac2u\Bigr)^{-3/2}&2y\le z\le y^2,\cr 1&z\ge
y^2.\end{cases}$$
\item If $y>\sqrt{x}$, then
$$H(x,y,z)\asymp\begin{cases}\displaystyle H\Bigl(x,\frac xz,\frac xy\Bigr)&\frac xy\ge\frac
xz+1,\cr \eta x&\text{else}.\end{cases}$$
\end{enumerate}
\end{theorem}

When the interval $(y,z]$ is relatively short, Tenenbaum established
an asymptotic formula for $H(x,y,z)$.

\begin{theorem}[Tenenbaum \cite{ten}]\label{thm1b} If $z\le\sqrt{x}$ and $\xi\to\infty$, then
$$H(x,y,z)\sim\eta x\quad(y\to\infty,\;z-y\to\infty).$$
\end{theorem}

A natural generalization of $H(x,y,z)$ arises from restricting the
range of $n$ to be some subset of the natural numbers $\mathscr{A}$.
To this end we define
$$H(x,y,z;\mathscr{A})=\lvert\{n\in[0,x]\cap\mathscr{A}:\exists
d|n\;\text{with}\;y<d\le z\}\rvert.$$ If $\mathscr{A}$ is reasonably
well-distributed in arithmetic progressions, then a simple heuristic
shows that we should have
$$H(x,y,z;\mathscr{A})\approx\frac{|\mathscr{A}\cap[0,x]|}xH(x,y,z).$$
In the case that $\mathscr{A}$ is an arithmetic progression Ford,
Khan, Shparlinski and Yankov~\cite{fksy} obtained upper bounds on
$H(x,y,z;\mathscr{A})$. In the present paper we focus on the special
and important case when $\mathscr{A}=P_s:=\{p+s:p\;\text{prime}\}$
for fixed $s\neq 0$. It is well-known that $P_s$ is well-distributed
in arithmetic progressions $a\pmod q$ with $(a-s,q)=1$. Making this
precise using sieving arguments and combining it with the methods
developed in \cite{kf}\;can lead to bounds on $\hs$ of the expected
order of magnitude. The upper bounds were settled by Ford in
\cite{kf}. We state below a short interval version of Theorem 6 in
\cite{kf}; for a proof of it see the proofs of Theorem 6 and Lemma
6.1 in \cite{kf}.

\begin{theorem}[Ford \cite{kf}]\label{thm2} Fix $s\in\SZ\setminus\{0\}$. Let $2\le y\le\sqrt{x}$, $y+1\le z\le
x$ and $x(\log z)^{-10}\le\Delta\le x$. Then
$$\hsd\ll_s\begin{cases}\displaystyle\frac\Delta x\frac{H(x,y,z)}{\log x}&z\ge
y+(\log y)^{2/3},\cr \cr \displaystyle\frac\Delta{\log
x}\sum_{\substack{y<d\le z\\(d,s)=1}}\frac1{\phi(d)}&z\le y+(\log
y)^{2/3}.\end{cases}$$
\end{theorem}

\begin{rk} The reason that the upper bound in Theorem
\ref{thm2}\;has this particular shape is due to our incomplete
knowledge about the sum $\sum_{y<d\le z}\frac1{\phi(d)}$ when the
interval $(y,z]$ is very short. The main theorem in
\cite{Sita}\;implies that
$$\sum_{y<d\le
z}\frac1{\phi(d)}\asymp\log(z/y)\quad(z\ge y+(\log y)^{2/3}),$$
whereas standard conjectures on Weyl sums would yield that
\be\label{e0}\sum_{y<d\le z}\frac1{\phi(d)}\asymp\log(z/y)\quad(z\ge
y+\log\log y).\ee The range of $y$ and $z$ in \eqref{e0}\;is the
best possible one can hope for, since it is well-known that the
order of $n/\phi(n)$ can be as large as $\log\log n$ if $n$ has many
small prime factors.
\end{rk}

In general, lower bounds on $H(x,y,z;P_s)$ are more difficult
because they rely on more precise knowledge about the distribution
of primes in arithmetic progressions, which is a notoriously
difficult problem. A special case was worked out by Ford.

\begin{theorem}[Ford \cite{kf}]\label{thm3} For fixed $s,a,b$ with $s\in\SZ\setminus\{0\}$
and $0\le a<b\le1$ we have $$H(x,x^a,x^b;P_s)\gg_{s,a,b}\frac x{\log
x}.$$
\end{theorem}

The purpose of this paper is to provide lower bounds on
$H(x,y,z;P_s)$ in a broader range of the parameters $y$ and $z$. We
split our results according to the range of the parameter
$\eta=\log(z/y)$. For small values of $\eta$ lower bounds on
$H(x,y,z;P_s)$ depend heavily on inequalities of the form
\be\label{wgrh}\pi(x;q,a)\ge\frac{cx}{\phi(q)\log
x}\quad\text{for}\quad(a,q)=1\ee for some $c>0$, uniformly in some
range of $q$ with a possible `small' exceptional set, namely reverse
Brun-Titchmarsh inequalities. Such results have been proven by
Alford, Granville and Pomerance \cite{carnum}\;and
Harman\;\cite{Harman}. Also, Bombieri, Friedlander and Iwaniec
proved in \cite{bfi1}\;an asymptotic formula for
$$\sum_{\substack{q\le Q\\(q,a)=1}}\pi(x;q,a),$$ when $Q\le x^{1-\epsilon}$ and $a$ is fixed.
Combining these results with the arguments leading to Theorem
\ref{thm1b}\;we show the following theorem. Here and for the rest of
this paper $x_0(\cdot)$ denotes a sufficiently large positive
constant which depends only on the parameters given, e.g. $x_0(s)$,
and its meaning might change from statement to statement.

\begin{theorem}[Small values of $\eta$]\label{thmsmall1} Fix
$s\in\SZ\setminus\{0\}$. Let $3\le y+1\le z\le x$ with
$y\le\sqrt{x}$ and $\{y<d\le z:(d,s)=1\}\neq\emptyset$.
\begin{enumerate}
\item Let $\epsilon>0$. If $x\ge x_0(s,\epsilon)$, $z\le x^{5/12-\epsilon}$ and
$$y+\log\log y\le z\le y+\frac y{(\log y)^2},$$ then
\be\label{thsm}\hs\gg\begin{cases}\displaystyle \frac{H(x,y,z)}{\log
x}& z\ge y+(\log y)^{2/3},\cr \cr \displaystyle\frac x{\log
x}\sum_{\substack{y<d\le z\\(d,s)=1}}\frac1{\phi(d)}&z\le y+(\log
y)^{2/3}\end{cases}\ee with the implied constant depending on $s$
and $\epsilon$. If, in addition, $(z-y)/\log\log y\to\infty$ as
$y\to\infty$, then
$$\hs\sim_{\epsilon,s}\begin{cases}\displaystyle f(s)\frac{315\zeta(3)}{2\pi^4}\frac{\eta x}{\log
x}&\text{if}\;\displaystyle\frac{z-y}{(\log y)^{2/3}}\to\infty,\cr
\cr \displaystyle\frac x{\log x}\sum_{\substack{y<d\le
z\\(d,s)=1}}\frac1{\phi(d)}&\text{otherwise},\end{cases}$$ as
$y\to\infty$, where $f(s)=\prod_{p|s}\frac{(p-1)^2}{p^2-p+1}$.
\item If $x\ge x_0(s)$, $z\le x^{0.472}$ and $$y+\exp\{4.532(\log
y)^{1/4}\}\le z\le y+\frac y{(\log y)^2},$$ then \eqref{thsm}\;holds
with the implied constant depending on $s$.
\item If \eqref{wgrh}\;holds for some $c>0$, uniformly in $q\le Q$ for some $Q=Q(x)\le\sqrt{x}$, $x\ge x_0(s,c)$ and
$$z\le y+\frac y{(\log y)^2},$$ then \eqref{thsm}\;is valid for $z\le
Q$ with the implied constant depending on $s$ and $c$.
\item Let $B\ge 2$ be fixed. If $$z\ge y+\frac{y}{(\log y)^B}\quad\text{and}\quad\xi\to\infty,$$
then $$\hs\sim_{s,B}f(s)\frac{315\zeta(3)}{2\pi^4}\frac{\eta x}{\log
x}\quad(y\to\infty).$$
\end{enumerate}
\end{theorem}

For intermediate and large values of $\eta$ we need results about
primes in arithmetic progressions {\it on average}\;in order to
control error terms coming from the linear sieve. The most famous
such result is the Bombieri-Vinogradov theorem \cite[p. 161]{dav}.
This theorem allows one to get the expected order of $\hs$ for $y\le
x^{1/2-\epsilon}$. To go beyond this threshold we make use of
Theorem 9 in \cite{bfi1}.

\begin{theorem}[Intermediate and large values of $\eta$; short intervals]\label{thminter1}
Fix $s\in\SZ\setminus\{0\}$ and $B\ge2$. Let $x\ge x_0(s,B)$,
$x(\log x)^{-B}\le\Delta\le x$ and $3\le y+1\le z\le x$ with
$\{y<d\le z:(d,s)=1\}\neq\emptyset$, $y\le\sqrt{x}$ and
$$z\ge y+\frac y{(\log y)^B}.$$ Then
$$\hsd\gg_{s,B}\frac\Delta x\frac{H(x,y,z)}{\log
x}.$$
\end{theorem}

We may combine Theorems \ref{thm2}\;and \ref{thminter1}\;with an
argument given in \cite{kf}\;to obtain the expected order of $\hs$
in the full range of the parameters $y$ and $z$, when $\eta\ge(\log
y)^{-B}$ for some fixed $B\ge2$.

\begin{theorem}[Intermediate and large values of
$\eta$]\label{thminter2} Fix $s\in\SZ\setminus\{0\}$ and $B\ge2$.
Let $x\ge x_0(s,B)$ and $3\le y+1\le z\le x$ with $\{y<d\le
z:(d,s)=1\}\neq\emptyset$ and
$$z\ge y+\frac y{(\log y)^B}.$$ Then
$$\hs\asymp_{s,B}\frac{H(x,y,z)}{\log
x}.$$
\end{theorem}

Finally, when $\eta$ is very large we are able to establish an
asymptotic formula for $\hs$, similar to the one given for
$H(x,y,z)$ in Theorem 21(iv) of \cite{hall_ten2}.

\begin{theorem}[Very large values of $\eta$]\label{thmlarge} Let
$s\in\SZ\setminus\{0\}$. If $2\le y\le z\le x$, then
$$\hs=\frac x{\log x}\biggl(1+O_s\biggl(\frac{\log y}{\log
z}\biggr)\biggr).$$
\end{theorem}

\medskip\medskip

\subsection*{Shifted primes in the multiplication table.} A
straightforward application of Theorem \ref{thminter2}\;is to the
multiplication table problem. This problem, which was first posed by
Erd\H os \cite{erd1},\cite{erd2}, is to count the number of distinct
integers of the form $ab$ with $1\le a,b\le N$, namely to estimate
the quantity
$$A(N):=|\{ab:1\le a,b\le N\}|.$$ A related question is to estimate
$$A(N;P_s):=|\{ab\in P_s:1\le a,b\le N\}|,$$ that is how many shifted primes
appear in the multiplication table. The order of $A(N)$ was
determined by Ford in \cite{kf}, where he proved that
$$A(N)\asymp\frac{N^2}{(\log N)^\delta(\log\log N)^{3/2}}.$$ This
follows by the elementary inequalities
$$H\Bigl(\frac{N^2}2,\frac N2,N\Bigr)\le
A(N)\le\sum_{m\ge0}H\Bigl(\frac{N^2}{2^m},\frac N{2^{m+1}},\frac
N{2^m}\Bigr)$$ and Theorem \ref{thm1}. Similarly, using Theorem
\ref{thminter2}\;we establish the order of magnitude of $A(N;P_s)$.

\begin{corollary} If $N\ge N_0(s)$, then
$$A(N;P_s)\asymp_s\frac{A(N)}{\log N}.$$
\end{corollary}


\section{Background material}

\textbf{Notation.} We make use of some standard notation. If $a(n)$,
$b(n)$ are two arithmetic functions, then we denote with $a*b$ their
Dirichlet convolution. Furthermore, for $n\in\SN$ and $1\le y\le z$
we put $\omega(n;y,z)=|\{p\;{\rm prime}:p|n,\;y<p\le z\}|$ and
$\Omega(n;y,z)=\sum\{a:p^a\|n,\;y<p\le z\}$, where $p^a\|n$ means
that $p^a|n$ and $p^{a+1}\nmid n$. Also, for brevity let
$\omega(n;z)=\omega(n;1,z)$ and $\Omega(n;z)=\Omega(n;1,z)$. For
$n\in\SN$ we use $P^+(n)$ and $P^-(n)$ to denote the largest and
smallest prime factor of $n$, respectively, with the notational
conventions that $P^+(1)=0$ and $P^-(1)=+\infty$. Given $1\le y<z$,
$\mathscr{P}(y,z)$ denotes the set of all integers $n$ such that
$P^+(n)\le z$ and $P^-(n)>y$. In addition, $\pi(x;q,a)$ stands for
the number of primes up to $x$ in the arithmetic progression $a\pmod
q$. Lastly, for a Dirichlet character $\chi$, $N(\sigma,V,\chi)$
denotes the number of zeros $\rho=\beta+i\gamma$ of its associated
$L$-function with $|\gamma|\le V$ and $\beta\ge\sigma$.

\medskip

In this section we state various preliminary results that are needed
in order to prove Theorems \ref{thmsmall1}, \ref{thminter1},
\ref{thminter2}\;and \ref{thmlarge}. First, we list a series of
results on primes in arithmetic progressions. We start with a lemma
which is a direct corollary of Theorem 2.1 in \cite{carnum}.

\begin{lemma}\label{prl1} Let $\epsilon\in(0,1/12)$. There exists $x_\epsilon\ge1$ such that
for every $x\ge x_\epsilon$, there is a set
$\mathcal{D}_\epsilon(x)\subset\SN\cap[\log x,x]$ with
$|\mathcal{D}_\epsilon(x)|\ll_\epsilon1$ such that for every
$(a,q)=1$ with $q\le x^{5/12-\epsilon}$,
$$\Bigl\lvert\pi(x;q,a)-\frac{{\rm
li}(x)}{\phi(q)}\Bigr\rvert\le\epsilon\frac{{\rm li}(x)}{\phi(q)},$$
with the possible exception of
$q\in\mathcal{MD}_\epsilon(x)=\{md:m\in\SN,d\in\mathcal{D}_\epsilon(x)\}$.
\end{lemma}

Harman \cite{Harman}, allowing a larger set of exceptional moduli,
gave a variation of Lemma \ref{prl1}. His starting point is the
following result.

\begin{lemma}\label{prl2} Given $\epsilon>0$, there are constants
$K(\epsilon)\ge2$ and $c(\epsilon)>0$ such that if
$K(\epsilon)<q<x^{0.472}$ and for every $d|q$ with $\chi$ a
primitive character $\pmod d$ we have
$$L(\sigma+it,\chi)\neq0\;\;\;\text{for}\;\sigma>1-\frac1{(\log
q)^{3/4}},\;|t|\le\exp\{\epsilon(\log q)^{3/4}\},$$ then for any $a$
with $(a,q)=1$ we have
$$\pi(x;q,a)\ge\frac{c(\epsilon)x}{\phi(q)\log x}.$$
\end{lemma}

Using Lemma \ref{prl2}\;along with estimates on averages of
$N(\sigma,V,\chi)$ Harman showed a variation of Lemma \ref{prl1}.
The main part of the argument is given in \cite{Harman}, but the
result is not stated explicitly; we state it and prove it here for
the sake of completeness.

\renewcommand{\labelenumi}{(\arabic{enumi})}

\begin{lemma}\label{prl3} There exist absolute positive constants $c_1$, $c_2$
and $x_0$ so that for all $x\ge x_0$ there is a set
$\mathcal{E}(x)\subset\SN\cap[\log x,x]$ satisfying the following:
\begin{enumerate}\item $|\mathcal{E}(x)|\le\exp\{3.641(\log x)^{1/4}\}$;
\item $|\mathcal{E}(x)\cap[1,\exp\{c_1(\log x)^{3/4}\}]|\ll1$;
\item For
every $(a,q)=1$ with $q\le x^{0.472}$ we have
$$\pi(x;q,a)\ge\frac{c_2x}{\phi(q)\log x},$$ with the possible
exception of
$q\in\mathcal{ME}(x)=\{me:m\in\SN,e\in\mathcal{E}(x)\}$.
\end{enumerate}
\end{lemma}

\begin{proof} Set $W=(0.4166\log x)^{3/4}$. From \cite[p. 93, 95]{dav}\;there
is an absolute constant $c_1$ such that there is at most one
primitive character $\chi_1$ to a modulus $q_1\le V=\exp\{c_1(\log
x)^{3/4}\}$ whose $L$-function has a zero $\rho$ with $|{\rm
Im}(\rho)|\le V$ and ${\rm Re}(\rho)>1-1/W$. By \cite[p. 96]{dav},
this exceptional modulus $q_1$ satisfies $q_1\ge\log x$. In
addition, Montgomery showed in \cite{mont}\;that
\be\label{pre1}\sum_{q\le Q}\quad\sideset{}{^*}\sum_{\chi\;({\rm
mod}\;q)}N(\sigma,V,\chi)\ll (Q^2V)^{2(1-\sigma)/\sigma}(\log
QV)^{13}\;\;\;(4/5\le\sigma\le1),\ee where $\sideset{}{^*}\sum$
means that the sum runs over primitive characters only. Inequality
\eqref{pre1}\;with $Q=x^{0.472}$ and $\sigma=1-1/W$ yields that
$N(\sigma,V,\chi)=0$ for all primitive characters to every moduli
$q\le x^{0.472}$ with at most $\exp\{3.64094(\log x)^{1/4}\}$
exceptions. Call this exceptional set $\mathcal{E}_1(x)$. This set
contains no elements $\le\log x$ and at most one element $\le V$, by
the discussion in the beginning of the proof. Next, applying Lemma
\ref{prl1}\;with $\epsilon_0=2/3\times10^{-4}$ we obtain a set
$\mathcal{D}_{\epsilon_0}(x)\subset[\log x,x]$ with boundedly many
elements and the property that if $q\le x^{0.4166}$ and
$q\notin\mathcal{MD}_{\epsilon_0}(x)$, then
\be\label{prl3e2}\pi(x;q,a)\ge(1-\epsilon_0)\frac x{\phi(q)\log
x}\quad{\rm for}\quad(a,q)=1.\ee Set
$$\mathcal{E}(x)=\mathcal{E}_1(x)\cup\mathcal{D}_{\epsilon_0}(x).$$ Clearly,
conditions (1) and (2) hold for $\mathcal{E}(x)$. Also, if
$q\le x^{0.4166}$ is such that $q\notin\mathcal{ME}(x)$, then (3)
holds by \eqref{prl3e2}. Finally, if $q\in[x^{0.4166},x^{0.472}]$
and $q\notin\mathcal{ME}(x)$, then the hypothesis of Lemma
\ref{prl2}\;is met and we deduce (3). This completes the proof of
the lemma.
\end{proof}

Below we state the Brun-Titchmarsh inequality \cite[Theorem
3.7]{Halb}.

\begin{lemma}\label{prl4} Uniformly in $1\le q<y\le x$ and $(a,q)=1$ we
have that $$\pi(x;q,a)-\pi(x-y;q,a)\ll\frac y{\phi(q)\log(2y/q)}.$$
\end{lemma}

In addition, we will need a generalization of Lemma \ref{prl4},
which is an easy application of the results and methods in
\cite{nair_ten}. Let $\mathcal{M}$ denote the class of functions
$F:\SN\to[0,+\infty)$ for which there exist constants $A_F$ and
$B_{F,\epsilon}$, $\epsilon>0$, such that
$$F(nm)\le\min\{A_F^{\Omega(m)},B_{F,\epsilon}m^\epsilon\}F(n)$$ for all $(m,n)=1$ and all $\epsilon>0$.

\begin{lemma}\label{l2} Let $F\in\mathcal{M}$, $a\in\SZ\setminus\{0\}$ and $1\le q\le h\le x$ such that
$(a,q)=1$ and $x>|a|$. If $q\le x^{1-\epsilon}$ and $\frac
hq\ge(\frac{x-a}q)^\epsilon$ for some $\epsilon>0$, then
$$\sum_{\substack{x-h<p\le x\\p\equiv\;a({\rm
mod}\;q)}}F\Bigl(\frac{p-a}q\Bigr)\ll_{a,\epsilon,F}\frac
h{\phi(q)(\log x)^2}\sum_{n\le x}\frac{F(n)}n;$$ the implied
constant depends on $F$ only via the constants $A_F$ and
$B_{F,\alpha},\;\alpha>0$.
\end{lemma}

\begin{proof} Observe that it suffices to show the
lemma for the function $\widetilde{F}$ defined for $n=2^rm$ with
$(m,2)=1$ by
$$\widetilde{F}(n)=\min\{A_F^r,\min_{\epsilon>0}(B_{F,\epsilon}2^{r\epsilon})\}F(m).$$
We have that $\widetilde{F}\in\mathcal{M}$ with parameters $A_F$ and
$B_{F,\alpha}^2,~\alpha>0$. Without loss of generality we may assume
that $\widetilde{F}(1)=1$. Also, suppose that $x\ge
x_0(\epsilon,a,F)$, where $x_0(a,\epsilon,F)$ is a sufficiently
large constant; otherwise, the result is trivial. Put
$$q_1=\begin{cases}q&{\rm if}\;2|aq\cr 2q,&{\rm if}\;2\nmid
aq,\end{cases}$$ and let $X=(x-a)/q_1$ and $H=h/q_1$. Note that if
$p\equiv a\pmod q$ and $p>2$, then $p\equiv a\pmod{q_1}$. So if we
set $p=q_1m+a$ for $p>2$, then
\be\begin{split}\sum_{\substack{x-h<p\le x\\p\equiv\;a({\rm
mod}\;q)}}\widetilde{F}\Bigl(\frac{p-a}q\Bigr)&\le\sum_{\substack{X-H<m\le
X\\P^-(q_1m+a)>\sqrt{X}}}\widetilde{F}\Bigl(\frac{q_1}qm\Bigr)+\sum_{\substack{X-H<m\le
X\\3\le q_1m+a\le\sqrt{X}}}\widetilde{F}\Bigl(\frac{q_1}qm\Bigr)+O_{a,F}(1)\nonumber\\
&\ll_{a,F}\sum_{\substack{X-H<m\le
X\\P^-(q_1m+a)>\sqrt{X}}}\widetilde{F}(m)+\sum_{\substack{X-H<m\le
X\\m\le\sqrt{X}-a}}\widetilde{F}(m)+1,\end{split}\ee since
$q_1/q\in\{1,2\}$ and $\widetilde{F}(2m)\ll_F\widetilde{F}(m)$ for
all $m\in\SN$. Let $F_1(n)=\widetilde{F}(n)$ and $F_2(n)$ be the
characteristic function of integers $n$ such that $P^-(n)>\sqrt{X}$.
Let $Q_1(x)=x$, $Q_2(x)=q_1x+a$ and $Q=Q_1Q_2$. Also, if
$P(x)\in\SZ[x]$, then let $\rho_P(m)$ be the number of solution of
the congruence $P(x)\equiv0\pmod m$. By Corollary 3 in
\cite{nair_ten}, we have that
\be\label{l2e1}\begin{split}\sum_{\substack{X-H<m\le
X\\P^-(q_1m+a)>\sqrt{X}}}\widetilde{F}(m)&=\sum_{X-H<m\le X}F_1(m)F_2(mq_1+a)\\
&\ll_{a,\epsilon,F} H\prod_{p\le
X}\Bigl(1-\frac{\rho_Q(p)}p\Bigr)\prod_{j=1}^2\sum_{n\le
X}\frac{F_j(n)\rho_{Q_j}(n)}n\\
&\ll_{a,\epsilon}\frac h{\phi(q)}\frac1{\log^2x}\sum_{n\le
X}\frac{\widetilde{F}(n)}n,\end{split}\ee since $q\le
x^{1-\epsilon}$ and the discriminant of $Q$ depends only on $a$.
Also, if the sum
$$\sum_{\substack{X-H<m\le X\\m\le\sqrt{X}-a}}\widetilde{F}(m)$$ is non-zero, then $H\ge X/2$. In this case
Corollary 3 in \cite{nair_ten}\;implies that
$$\sum_{\substack{X-H<m\le X\\m\le\sqrt{X}-a}}\widetilde{F}(m)\ll_{a,\epsilon,F}\frac{\sqrt{X}}{\log X}\sum_{n\le
X}\frac{\widetilde{F}(n)}n \ll_{a,\epsilon}\frac
h{q\log^2x}\sum_{n\le X}\frac{\widetilde{F}(n)}n,$$ which together
with~\eqref{l2e1} completes the proof of the lemma.
\end{proof}

Using Lemma \ref{l2}\;we prove the following estimate.

\begin{lemma}\label{l2b} Let $1\le v\le v_0<2$, $a\in\SZ\setminus\{0\}$,
$1\le q\le x$ and $3/2\le y\le(x-a)/q$ with $(a,q)=1$ and $x>|a|$.
If $q\le x^{1-\epsilon}$ for some $\epsilon>0$, then
$$\sum_{\substack{p\le x\\p\equiv\;a({\rm
mod}\;q)}}v^{\Omega(\frac{p-a}q;y)}\ll_{a,\epsilon,v_0}\frac
x{\phi(q)\log x}(\log y)^{v-1}.$$
\end{lemma}

\begin{proof} We may assume that $x\ge x_0(a,\epsilon,v_0)$, where $x_0(a,\epsilon,v_0)$
is a sufficiently large constant. Let $X=(x-a)/q$ and write
$v^{\Omega(n;y)-\omega(n;y)}=(1*b)(n)$, where $b$ is the
multiplicative function that satisfies
$$b(p^l)=\begin{cases}0&\text{if}\;l=1\;{\rm or}\;p>y,\cr
v^{l-2}(v-1)&\text{if}\;l\ge2\;{\rm and}\;p\le y.\end{cases}$$ Then
$$v^{\Omega(n;y)}=v^{\omega(n;y)}\sum_{kf=n}b(k)\le\sum_{kf=n}b(k)v^{\omega(k;y)}v^{\omega(f;y)}$$
and consequently \be\label{prl10}\sum_{\substack{p\le
x\\p\equiv\;a({\rm mod}\;q)}}v^{\Omega(\frac{p-a}q;y)}\le\sum_{k\le
X}v^{\omega(k;y)}b(k)\sum_{\substack{p\le x\\p\equiv\;a({\rm
mod}\;qk)}}v^{\omega(\frac{p-a}{qk};y)}.\ee If $k\le\sqrt{X}$, then
$kq\le x^{1-\epsilon/3}$. So Lemma \ref{l2}\;implies that
$$\sum_{\substack{p\le x\\p\equiv\;a({\rm
mod}\;qk)}}v^{\omega(\frac{p-a}{qk};y)}\ll_{a,\epsilon}\frac{x(\log
y)^{v-1}}{\phi(kq)\log x}.$$ If $k>\sqrt{X}$, then
$$\sum_{\substack{p\le x\\p\equiv\;a({\rm
mod}\;qk)}}v^{\omega(\frac{p-a}{qk};y)}\le\sum_{m\le
X/k}v^{\omega(m)}\ll_{a,\epsilon}\frac{x(\log X)^{v-1}}{kq},$$ by
Theorem 01 in \cite{hall_ten2}. Hence the right hand side of
\eqref{prl10}\;is $$\ll_{a,\epsilon}\frac{x(\log
y)^{v-1}}{\phi(q)\log
x}\sum_{k\le\sqrt{X}}\frac{v^{\omega(k;y)}b(k)}{\phi(k)}+\frac{x(\log
X)^{v-1}}{qX^{\alpha/2}}\sum_{\sqrt{X}<k\le
X}\frac{v^{\omega(k;y)}b(k)k^\alpha}k\ll_{a,\epsilon,v_0}\frac{x(\log
y)^{v-1}}{\phi(q)\log x},$$ provided that $0<\alpha<1/2$ and
$2^{1-\alpha}>v_0$, which completes the proof.
\end{proof}

We complete the results about primes in arithmetic progressions with
the following estimate.

\begin{lemma}\label{prl5} Let $a\in\SZ\setminus\{0\}$, $\epsilon>0$ and $A>0$. There exists $B=B(A)$ such that if $R\le
x^{1/10-\epsilon}$ and $QR<x(\log x)^{-B}$, then
$$\sum_{\substack{r\le R\\(r,a)=1}}\biggl\lvert\sum_{\substack{q\le Q\\(q,a)=1}}\Bigl(\pi(x;rq,a)-\frac{{\rm li}(x)}{\phi(rq)}\Bigr)\biggr\rvert
\ll_{a,\epsilon,A}\frac x{(\log x)^A}.$$
\end{lemma}

\begin{proof} Use Theorem 9 in \cite{bfi1}\;plus partial summation.
\end{proof}

We need an estimate on the summatory function of the reciprocals of
Euler's $\phi$ function and other closely related quantities. Such a
result was proved by Sitaramachandra \cite{Sita}. Using the methods
of \cite{Sita}\;we extend this result according to the needs of this
paper.

\begin{lemma}\label{prl7} Let $a\in\SN$, $s\in\SZ$ and $x\ge1$ such that $1\le|s|\le x$. Then
\be\begin{split}\sum_{\substack{n\le
x\\(n,s)=1}}\frac{\phi(a)}{\phi(an)}=\frac{315\zeta(3)}{2\pi^4}\frac{\phi(s)}{|s|}g(as)\Bigl(\log
x+\gamma-\sum_{p\nmid as}\frac{\log p}{p^2-p+1}+\sum_{p|s}\frac{\log
p}{p-1}\Bigr)\nonumber\\
+O\Bigl(\tau(s)\frac{a|s|}{\phi(as)}\frac{(\log2x)^{2/3}}x\Bigr),\end{split}\ee
where $g(as)=\prod_{p|as}\frac{p(p-1)}{p^2-p+1}$.
\end{lemma}

\begin{proof} Since the proof of this part is along the same lines with the proof of the main
result in \cite{Sita}, we simply sketch it. Let $P(x)=\{x\}-1/2$,
where $\{x\}$ denotes the fractional part of $x$. Then using the
estimate
$$\sum_{n\le x}\frac{P(x/n)}n\ll(\log2x)^{2/3},$$ which was proved
in \cite[p. 98]{wal}, along with an argument similar to the one
leading to Lemma 2.2 in \cite{Sita}, we find that
\be\label{pre3}\sum_{\substack{n\le
x\\(n,m)=1}}\frac{\mu^2(n)}{\phi(n)}P(x/n)\ll\frac{|m|}{\phi(m)}(\log2x)^{2/3}\ee
for every $m\in\SZ\setminus\{0\}$. Also, by the Euler-McLaurin's
summation formula we have \be\label{pre4}\sum_{n\le x}\frac1n=\log
x+\gamma-\frac{P(x)}x+O\Bigl(\frac1{x^2}\Bigr).\ee Observe that the
arithmetic function $n\to\phi(a)/\phi(an)$ is multiplicative. In
particular, we have that
\be\label{pre3b}\frac{\phi(a)}{\phi(an)}=\sum_{\substack{kf=n\\(k,a)=1}}\frac{\mu^2(k)}{k\phi(k)f}.\ee
Using relations \eqref{pre3}, \eqref{pre4}\;and \eqref{pre3b}\;and
estimating the error terms as in \cite{Sita}\;gives us that
\be\begin{split}\sum_{\substack{n\le
x\\(n,s)=1}}\frac{\phi(a)}{\phi(an)}&=\sum_{\substack{k\le
x\\(k,as)=1}}\frac{\mu^2(k)}{k\phi(k)}\sum_{\substack{f\le
x/k\\(f,s)=1}}\frac1f=\sum_{d|s}\frac{\mu(d)}d\sum_{\substack{k\le
x/d\\(k,as)=1}}\frac{\mu^2(k)}{k\phi(k)}\sum_{l\le x/kd}\frac1l\\
&=\sum_{d|s}\frac{\mu(d)}d\sum_{\substack{k\le
x/d\\(k,as)=1}}\frac{\mu^2(k)}{k\phi(k)}\Bigl(\log\frac{x/d}k+\gamma-\frac
k{x/d}P\Bigl(\frac{x/d}k\Bigr)+O\Bigl(\frac{k^2}{(x/d)^2}\Bigr)\Bigl)\\
&=\sum_{d|s}\frac{\mu(d)}d\sum_{\substack{k=1\\(k,as)=1}}^\infty\frac{\mu^2(k)}{k\phi(k)}\Bigl(\log\frac{x/d}k+\gamma\Bigr)
+O\Bigl(\frac{\tau(s)a|s|}{\phi(as)}\frac{(\log2x)^{2/3}}x\Bigr),\nonumber\end{split}\ee
since $|s|\le x$. Finally, a simple calculation and the identity
$$\sum_{k=1}^\infty\frac{\mu^2(k)}{k\phi(k)}=\frac{315\zeta(3)}{2\pi^4}$$ complete the proof.
\end{proof}

The following result is known as the `fundamental lemma' of sieve
methods. It has appeared in the literature in several different
forms (see for example \cite[Theorem 2.5, p. 82]{Halb}). We need a
version of it that can be found in \cite{fi}\;and \cite{iwan1}.

\begin{lemma}\label{fundle} Let $D\ge2$, $D=Z^v$ with $v\ge3$.
\begin{enumerate}\item[(a)] Fix $\kappa>0$. There exist two sequences $\{\lambda^+(d)\}_{d\le D}$, and
$\{\lambda^-(d)\}_{d\le D}$ such that $$|\lambda^\pm(d)|\le1,$$
$$\begin{cases}(\lambda^-*1)(n)=(\lambda^+*1)(n)=1&\text{if}\;P^-(n)>Z,\cr
(\lambda^-*1)(n)\le0\le(\lambda^+*1)(n)&\text{otherwise},\end{cases}$$
and, for any multiplicative function $\alpha(d)$ with $0\le
\alpha(p)\le\min\{\kappa,p-1\}$, $$\sum_{d\le
D}\lambda^\pm(d)\frac{\alpha(d)}d=\prod_{p\le
Z}\Bigl(1-\frac{\alpha(p)}p\Bigr)(1+O_\kappa(e^{-v})).$$
\item[(b)] There exists a sequence $\{\rho(d)\}_{d\le D}$ such that
\be\label{prop1}|\rho(d)|\le1,\ee
\be\label{prop2}\begin{cases}(\rho*1)(n)=1&\text{if}\;P^-(n)>Z,\cr
(\rho*1)(n)\le0&\text{otherwise},\end{cases}\ee and, for any
multiplicative function $\alpha(d)$ satisfying $0\le \alpha(p)\le
p-1$ and \be\label{iwancond}\prod_{y<p\le
w}\Bigl(1-\frac{\alpha(p)}p\Bigr)^{-1}\le\frac{\log w}{\log
y}\Bigl(1+\frac L{\log y}\Bigr)\quad(3/2\le y\le w),\ee we have
\be\label{prop3}\sum_{d\le D}\rho(d)\frac{\alpha(d)}d\gg\prod_{p\le
Z}\Bigl(1-\frac{\alpha(p)}p\Bigr),\ee provided that $D\ge D_0(L)$,
where $D_0(L)$ is a constant depending only on $L$.
\end{enumerate}
\end{lemma}

\begin{proof} (a) The result follows by \cite[Lemma 5, p. 732]{fi}.

\medskip

(b) The construction of the sequence $\{\rho(d)\}_{d\le D}$ and the
proof that it satisfies the desired properties is based on
\cite[Lemma 5]{fi}\;and \cite[Lemma 3]{iwan1}. We sketch the proof
below. Without loss of generality we may assume that $Z\notin\SN$.
Set $P(Z)=\prod_{p<Z}p$ and $\rho(d)=\mu(d){\bf 1}_A(d)$, where
${\bf 1}_A$ is the characteristic function of the set
$$A=\{d|P(Z):d=p_1\cdots
p_r,\;p_r<\cdots<p_1<Z,\;p_{2l}^3p_{2l-1}\cdots p_1<D\;(1\le l\le
r/2)\}.$$ By the proof of Lemma 5 in \cite{fi}, the sequence
$\{\rho(d)\}_{d=1}^\infty$ is supported in $\{d\in\SN:d<D\}$ and
satisfies \eqref{prop1}\;and \eqref{prop2}. Finally, by Lemma 3 in
\cite{iwan1}, there exists a function $h$, independent of the
parameters $D$, $Z$ and $L$, such that $$\sum_{d\le
D}\rho(d)\frac{\alpha(d)}d\ge\bigl(h(v)+O(e^{\sqrt{L}-v}(\log
D)^{-1/3})\bigr)\prod_{p<Z}\Bigl(1-\frac{\alpha(p)}p\Bigr)$$ for all
multiplicative functions $\alpha(d)$ that satisfy $0\le\alpha(p)\le
p-1$ and \eqref{iwancond}. In addition, $h$ is increasing and
$h(3)>0$, by \cite[p. 172-173]{iwan2}. This proves that
\eqref{prop3}\;holds too and completes the proof of the lemma.
\end{proof}

We now introduce some notation we will be utilizing later. For $a$
and $k$ in $\SN$ and $1\le y<z$ define
$$\tau(a)=|\{d\in\SN:d|n\}|,\quad\tau(a,y,z)=|\{d\in\SN:d|n,y<d\le z\}|$$ and
$$\tau_k(a)=\lvert\{(d_1,\dots,d_k)\in\SN^k:d_1\cdots d_k=a\}\rvert.$$ Moreover, for $\sigma>0$ let
$$\mathscr{L}(a;\sigma)=\{x\in\SR:\tau(a,e^x,e^{x+\sigma})\ge1\}=\bigcup_{d|a}[\log d-\sigma,\log d)$$ and
$$L(a;\sigma)={\rm meas}(\mathscr{L}(a;\sigma)),$$ where `${\rm
meas}$' denotes the Lebesgue measure on the real line. We note the
straightforward inequality
\be\label{pre2}L(ab;\sigma)\le\tau(a)L(b;\sigma)\quad\text{for}\quad(a,b)=1,\ee
which is item (ii) of Lemma 3.1 in \cite{kf}.

When $\eta$ is in the intermediate range of values, the basic result
we will use to bound $\hs$ from below is the following estimate.

\begin{lemma}\label{prl8}Let $\epsilon>0$, $B>0$, $x\ge1$, $3\le y+1\le z$ with $z\le x^{2/3}$ and
$\eta\in[(\log y)^{-B},\frac{\log y}{100}]$. Then
$$H(x,y,z)\asymp_{\epsilon,B}\frac x{\log^2y}\sum_{\substack{a\le
y^\epsilon\\\mu^2(a)=1}}\frac{L(a;\eta)}a.$$
\end{lemma}

The proof of Lemma \ref{prl8}\;can be found in \cite{kf}. Even
though this result\;is not stated explicitly, it is a direct
corollary of the methods there: see Theorem 1 and Lemmas 4.1, 4.2,
4.5, 4.8 and 4.9 in \cite{kf}. Also, we will need the following
result, which is Corollary 1 in \cite{kf}.

\begin{lemma}\label{cor1} Suppose $x_1, y_1, z_1, x_2, y_2, z_2$ are
real numbers with $2\le y_i+1\le z_i\le x_i\;(i=1,2)$,
$\log(z_1/y_1)\asymp\log(z_2/y_2)$, $\log y_1\asymp\log y_2$ and
$\log(x_1/z_1)\asymp\log(x_2/z_2)$. Then
$$\frac{H(x_1,y_1,z_1)}{x_1}\asymp\frac{H(x_2,y_2,z_2)}{x_2}.$$
\end{lemma}

Finally, we state a covering lemma, which a special case of Lemma
3.15 in \cite{fol}. Here for $I$ an interval of the real line we
denote by $rI$ the interval that has the same center as $I$ and $r$
times its diameter.

\begin{lemma}\label{prl9} Let $A=\bigcup_{n=1}^NI_n\subset\SR$, where the
sets $I_n$ are nonempty intervals of the form $[a,b)$. Then there
exists a subcollection $\{I_{i_m}\}_{m=1}^M$ of mutually disjoint
intervals such that
$$A\subset\bigcup_{m=1}^M3I_{i_m}.$$
\end{lemma}


\section{Small values of $\eta$}

In this section we give the proof of Theorem \ref{thmsmall1}. First,
we show an auxiliary result.

\renewcommand{\labelenumi}{(\alph{enumi})}

\begin{lemma}\label{svl1} Let $a\in\SZ\setminus\{0\}$, $x\ge2$ and $3\le Q_1+1\le Q_2\le2Q_1$ with
$Q_1\le\sqrt{x}$ and $\{Q_1<q\le Q_2:(q,a)=1\}\neq\emptyset$.
\begin{enumerate}\item Let $\epsilon\in(0,1/12)$. If $x\ge x_0(a,\epsilon)$ and $Q_1+\log\log Q_1\le Q_2\le
x^{5/12-\epsilon}$, then \be\label{svl1e}\sum_{\substack{Q_1<q\le
Q_2\\(q,a)=1}}\pi(x;q,a)\gg_{a,\epsilon}\frac x{\log
x}\sum_{\substack{Q_1<q\le Q_2\\(q,a)=1}}\frac1{\phi(q)}.\ee If, in
addition, $(Q_2-Q_1)/\log\log Q_1\to\infty$ as $Q_1\to\infty$, then
\be\label{svl1ee}\sum_{\substack{Q_1<q\le
Q_2\\(q,a)=1}}\pi(x;q,a)\sim_{a,\epsilon}\frac x{\log
x}\sum_{\substack{Q_1<q\le
Q_2\\(q,a)=1}}\frac1{\phi(q)}\;\;\;(Q_1\to\infty).\ee
\item If $x\ge x_0(a)$ and $Q_1+\exp\{4.532(\log Q_1)^{1/4}\}\le Q_2\le x^{0.472}$, then
\eqref{svl1e}\;holds with the implied constant depending only on
$a$.
\medskip
\item Let $B\ge2$. If $Q_2\ge Q_1+Q_1(\log Q_1)^{-B}$, then
$$\sum_{\substack{Q_1<q\le
Q_2\\(q,a)=1}}\pi(x;q,a)\sim_{a,B}f(a)\frac{315\zeta(3)}{2\pi^4}\frac
{x\log(Q_2/Q_1)}{\log x}.$$
\end{enumerate}
\end{lemma}

\begin{proof} (a) For every $\epsilon_1\in(0,\epsilon]$ and $x\ge
x_{\epsilon_1}$ Lemmas \ref{prl1}\;and \ref{prl4}\;imply that
\be\label{sv3}\begin{split}\sum_{\substack{Q_1<q\le
Q_2\\(q,a)=1}}\pi(x;q,a)&=(1+\epsilon_1\theta){\rm
li}(x)\sum_{\substack{Q_1<q\le
Q_2\\q\notin\mathcal{MD}_{\epsilon_1}(x)\\(q,a)=1}}\frac1{\phi(q)}+O\Bigl(\frac
x{\log x}\sum_{\substack{Q_1<q\le
Q_2\\q\in\mathcal{MD}_{\epsilon_1}(x)}}\frac1{\phi(q)}\Bigr)\\
&=(1+\epsilon_1\theta){\rm li}(x)\sum_{\substack{Q_1<q\le
Q_2\\(q,a)=1}}\frac1{\phi(q)}+O\Bigl(\frac x{\log
x}\sum_{\substack{Q_1<q\le
Q_2\\q\in\mathcal{MD}_{\epsilon_1}(x)}}\frac1{\phi(q)}\Bigr),\end{split}\ee
for some $|\theta|\le1$. Fix $d\in\mathcal{D}_{\epsilon_1}(x)$. If
$d\ge Q_2-Q_1$, then the interval $(Q_1/d,Q_2/d]$ contains at most
one integer and therefore \be\label{sv5}\sum_{Q_1/d<m\le
Q_2/d}\frac1{\phi(dm)}\ll\frac{\log\log Q_1}{Q_1}.\ee On the other
hand, if $d\le Q_2-Q_1$, then \be\label{sv5b}\sum_{Q_1/d<m\le
Q_2/d}\frac1{\phi(dm)}\ll\frac{\log\log Q_1}d\log(Q_2/Q_1).\ee Since
$d\ge\log x$ and $|\mathcal{D}_{\epsilon_1}(x)|\ll_{\epsilon_1}1$,
relations \eqref{sv5}\;and \eqref{sv5b}\;yield that
\be\label{sve5c}\sum_{\substack{Q_1<q\le
Q_2\\q\in\mathcal{MD}_{\epsilon_1}(x)}}\frac1{\phi(q)}\ll_{\epsilon_1}\frac{\log\log
Q_1}{Q_1}+\frac{\log\log Q_1}{\log x}\log(Q_2/Q_1).\ee Also,
\be\label{sve0}\sum_{\substack{Q_1<q\le
Q_2\\(q,a)=1}}\frac1{\phi(q)}\ge\sum_{\substack{Q_1<q\le
Q_2\\(q,a)=1}}\frac1q\gg_a\log(Q_2/Q_1)\asymp\frac{Q_2-Q_1}{Q_1},\ee
uniformly in $Q_1+1\le Q_2\le2Q_1$ with $\{Q_1<q\le
Q_2:(q,a)=1\}\neq\emptyset$. The above inequality, \eqref{sv3}\;and
\eqref{sve5c}\;imply that $$\sum_{\substack{Q_1<q\le
Q_2\\(q,a)=1}}\pi(x;q,a)=(1+\epsilon_1\theta){\rm
li}(x)\sum_{\substack{Q_1<q\le
Q_2\\(q,a)=1}}\frac1{\phi(q)}\Bigl(1+O_{a,\epsilon_1}\Bigl(\frac{\log\log
Q_1}{\log x}+\frac{\log\log Q_1}{Q_2-Q_1}\Bigr)\Bigr).$$ This proves
that \eqref{svl1ee}\;holds. Next, we show that \eqref{svl1e}\;holds.
Fix a large positive constant $M=M(\epsilon,a)$ with $M\ge
x_\epsilon$. If $Q_1\le M$ and $x$ is large enough, then
$$\sum_{\substack{Q_1<q\le Q_2\\(q,a)=1}}\pi(x;q,a)\ge\max_{\substack{Q_1<q\le Q_2\\(q,a)=1}}\pi(x;q,a)\gg_{a,\epsilon}\frac x{\log x}
\asymp_{a,\epsilon}\frac x{\log x}\sum_{\substack{Q_1<q\le
Q_2\\(q,a)=1}}\frac1{\phi(q)},$$ by our assumption that $\{Q_1<q\le
Q_2:(q,a)=1\}\neq\emptyset$ and the Prime Number Theorem for
arithmetic progressions \cite[p. 123]{dav}. So we may suppose that
$Q_1>M$. By \eqref{sv3}, \eqref{sve5c}\;and \eqref{sve0}\;with
$\epsilon_1=\epsilon$ we deduce that
\be\label{sve2}\sum_{\substack{Q_1<q\le
Q_2\\(q,a)=1}}\pi(x;q,a)\ge\frac x{2\log
x}\Bigl(\sum_{\substack{Q_1<q\le
Q_2\\(q,a)=1}}\frac1{\phi(q)}-C_{a,\epsilon}\frac{\log\log
Q_1}{Q_1}\Bigr)\ee for some positive constant $C_{a,\epsilon}$. We
separate two cases. If \be\label{sve3}\sum_{\substack{Q_1<q\le
Q_2\\(q,a)=1}}\frac1{\phi(q)}\ge2C_{a,\epsilon}\frac{\log\log
Q_1}{Q_1},\ee then \eqref{svl1e}\;holds by \eqref{sve2}. So assume
that \eqref{sve3}\;fails. Then, by \eqref{sve0}\;and our assumption
that $Q_2\ge Q_1+\log\log Q_1$, we have that
\be\label{sve4}\frac{\log\log
Q_1}{Q_1}\ll\log\frac{Q_2}{Q_1}\ll_a\sum_{\substack{Q_1<q\le
Q_2\\(q,a)=1}}\frac1q\le\sum_{\substack{Q_1<q\le
Q_2\\(q,a)=1}}\frac1{\phi(q)}\le2C_{a,\epsilon}\frac{\log\log
Q_1}{Q_1}.\ee Also, Lemma \ref{prl1}\;implies that
\be\label{sve5}\sum_{\substack{Q_1<q\le
Q_2\\(q,a)=1}}\pi(x;q,a)\ge\frac x{2\log x}\sum_{\substack{Q_1<q\le
Q_2\\(q,a)=1,\;q\notin\mathcal{MD}_{\epsilon}(x)}}\frac1{\phi(q)}\ge\frac
x{2\log x}\Bigl(\sum_{\substack{Q_1<q\le
Q_2\\(q,a)=1}}\frac1q-\sum_{\substack{Q_1<q\le
Q_2\\q\in\mathcal{MD}_{\epsilon}(x)}}\frac1q\Bigr).\ee By the
argument leading to~\eqref{sve5c} we find that
\be\label{sve4b}\sum_{\substack{Q_1<q\le
Q_2\\q\in\mathcal{MD}_{\epsilon}(x)}}\frac1q\ll_\epsilon\frac1{Q_1}+\frac{\log(Q_2/Q_1)}{\log
x}.\ee Inserting \eqref{sve4}\;and \eqref{sve4b}\;into
\eqref{sve5}\;proves \eqref{svl1e}\;in the case that
\eqref{sve3}\;does not hold too.

\medskip

(b) When $Q_1\le x^{0.41666}<x^{5/12}$ the result follows from part
(a). When $Q_1>x^{0.41666}$ note that $$Q_2-Q_1\ge\exp\{4.532(\log
Q_1)^{1/4}\}\ge\exp\{3.6411(\log x)^{1/4}\}.$$ So following a very
similar argument with the one given in part (a) and using Lemma
\ref{prl3}\;in place of Lemma \ref{prl1}\;we obtain the desired
result.

\medskip

(c) Apply Lemmas \ref{prl5}\;and \ref{prl7}.
\end{proof}

We are now in position to prove Theorem \ref{thmsmall1}.

\begin{proof}[Proof of Theorem \ref{thmsmall1}] First, assume that
$z\le y+y(\log y)^{-2}$. We treat all four parts of the theorem
simultaneously. Let $y_0$ be a large constant, possibly depending on
$s,B,\epsilon$ and $c$, the constant in \eqref{wgrh}, according to
the assumptions of each of the parts (a), (b) and (c). If $y\le
y_0$, then we trivially have that
$$\hs\ge\max_{\substack{y<d\le z\\(d,s)=1}}\pi(x-s;d,-s)\asymp_{y_0}\frac
x{\log x},$$ by our assumption that $\{y<d\le
z:(d,s)=1\}\neq\emptyset$ and the Prime Number Theorem for
arithmetic progressions \cite[p. 123]{dav}. So assume that $y>y_0$.
By the inclusion-exclusion principle, we have that
\be\begin{split}\label{sv0}\sum_{y<d\le
z}&\pi(x-s;d,-s)-\sum_{y<d_1<d_2\le
z}\pi(x-s;[d_1,d_2],-s)\le\hs\le\sum_{y<d\le
z}\pi(x-s;d,-s).\end{split}\ee Lemma \ref{prl4}\;then implies that
\be\label{sv1}\begin{split}\hs=\sum_{\substack{y<d\le
z\\(d,s)=1}}\pi(x-s;d,-s)+O\Bigl(\sum_{y<d_1<d_2\le z}\frac
x{\log(2x/[d_1,d_2])\phi([d_1,d_2])}\Bigl).\end{split}\ee In the sum
over $d_1$ and $d_2$ in the right hand side of \eqref{sv1}\;set
$m=(d_1,d_2)$ and $d_i=mt_i$, $i=1,2$. Since $t_1+1\le t_2$, we get
that $m\le d_2-d_1\le z-y$. Moreover, notice that
$$\log\frac{2x}{[d_1,d_2]}=\log\frac{2x}{t_1t_2m}\ge\log\frac{2xm}{z^2}\gg\frac{\log2m\log x}{\log
y},$$ uniformly in $y\le\sqrt{x}$. Therefore
\be\begin{split}\sum_{y<d_1<d_2\le
z}\frac1{\log(2x/[d_1,d_2])\phi([d_1,d_2])}&\ll\frac{(\log
y)(\log\log y)}{\log x}\sum_{m\le z-y}\frac
1{m\log2m}\sum_{y/m<t_1<t_2\le z/m}\frac
1{t_1t_2}\\
&\le\frac{(\log y)(\log\log y)}{\log x}\sum_{m\le
z-y}\frac1{m\log2m}\biggl(\sum_{y/m<t\le
z/m}\frac 1t\biggr)^2\\
&\ll\frac{\eta^2(\log y)(\log\log y)^2}{\log x}\ll\frac\eta{\log
x}\frac{(\log\log y)^2}{\log y},\nonumber\end{split}\ee which
combined with \eqref{sv1}\;yields that $$\hs=\sum_{\substack{y<d\le
z\\(d,s)=1}}\pi(x-s;d,-s)+O_s\Bigl(\frac{\eta x}{\log
x}\frac{(\log\log y)^2}{\log y}\Bigr).$$ The above estimate together
with Lemma \ref{svl1}\;and the inequality
$$\sum_{\substack{y<d\le z\\(d,s)=1}}\frac1{\phi(d)}\ge\sum_{\substack{y<d\le
z\\(d,s)=1}}\frac1d\gg_s\eta,$$ which holds uniformly in $y+1\le z$
with $\{y<d\le z:(d,s)=1\}\neq\emptyset$, completes the proof of
parts (a), (b) and (c) as well as of part (d) when $z\le y+y(\log
y)^{-2}$. It remains to show part (d)\;when $z>y+y(\log y)^{-2}$, in
which case $(\log y)^{-2}\ll\eta\ll(\log y)^{-\log4+1}$. First, by
\eqref{sv0}\;and Lemma \ref{svl1}(c), we have that
$$\hs\le\sum_{\substack{y<d\le
z\\(d,s)=1}}\pi(x-s;d,-s)+O_s(1)\sim_sf(s)\frac{315\zeta(3)}{2\pi^4}\frac{\eta
x}{\log x},$$ which proves the desired upper bound. For the lower
bound, let $\chi$ be the characteristic function of integers $n$
satisfying $$\Omega(n;y)\le L(y):=2\log\log y+\psi(y)(\log\log
y)^{1/2},$$ where $\psi(y)\to\infty$ as $y\to\infty$ and
$\psi(y)\ll(\log\log y)^{1/6}$. Then the inclusion-exclusion
principle and Lemma \ref{svl1}(c)\;imply that
\be\label{pfe11}\begin{split}\hs\ge\sum_{\substack{p+s\le
x\\\tau(p+s,y,z)\ge1}}\chi(p+s)&\ge\sum_{p+s\le
x}\chi(p+s)\Bigl(\sum_{\substack{d|p+s\\y<d\le
z}}1-\sum_{\substack{[d_1,d_2]|p+s\\y<d_1<d_2\le
z}}1\Bigr)\\
&\ge f(s)\frac{315\zeta(3)}{2\pi^4}\frac{\eta x}{\log
x}(1-o_s(1))-S-S^\prime,\end{split}\ee where
$$S:=\sum_{\substack{p+s\le x\\p\nmid
s}}(1-\chi(p+s))\sum_{\substack{d|p+s\\y<d\le z}}1\quad{\rm
and}\quad S^\prime:=\sum_{\substack{p+s\le x\\p\nmid
s}}\chi(p+s)\sum_{\substack{[d_1,d_2]|p+s\\y<d_1<d_2\le z}}1.$$ To
bound $S$ observe that for every $1\le v\le3/2$ we have that
\be\label{pfe6}\begin{split}S\le v^{-L(y)}\sum_{\substack{p+s\le
x\\p\nmid s}}v^{\Omega(p+s;y)}\sum_{\substack{d|p+s\\y<d\le z}}1
&\le v^{-L(y)}\sum_{\substack{y<d\le
z\\(d,s)=1}}v^{\Omega(d;y)}\sum_{\substack{p+s\le
x\\p\equiv-s\;({\rm
mod}\;d)}}v^{\Omega(\frac{p+s}d;y)}\\
&\ll_s\frac{xv^{-L(y)}(\log y)^{v-1}}{\log x}\sum_{y<d\le
z}\frac{v^{\Omega(d)}}{\phi(d)},\end{split}\ee by Lemma \ref{l2b}.
Writing $$\frac d{\phi(d)}=\sum_{k|d}\frac{\mu^2(k)}{\phi(k)}$$ and
using Theorem 04 in \cite{hall_ten2}\;we find that
\be\label{pfe8}\begin{split}\sum_{y<d\le z}\frac
{v^{\Omega(d)}}{\phi(d)}&=\sum_{k\le
z}\frac{\mu^2(k)v^{\Omega(k)}}{k\phi(k)}\sum_{y/k<f\le
z/k}\frac{v^{\Omega(f)}}f\\
&\ll\sum_{k\le\sqrt{y}}\frac{\mu^2(k)v^{\Omega(k)}}{k\phi(k)}\bigl(\eta(\log
(y/k))^{v-1}+(\log(y/k))^{v-3}\bigr)\\
&\quad\quad+\sum_{\sqrt{y}<k\le
z}\frac{\mu^2(k)v^{\Omega(k)}}{k\phi(k)}(\log y)^{v-1}\\
&\ll\eta(\log y)^{v-1}+\frac{(\log
y)^{v-1}}{y^{1/4}}\sum_{\sqrt{y}<k\le
z}\frac{\mu^2(k)v^{\Omega(k)}}{\sqrt{k}\phi(k)}\\
&\ll\eta(\log y)^{v-1},\end{split}\ee since $\eta\gg(\log y)^{-2}$.
Combining inequalities \eqref{pfe6}\;and \eqref{pfe8}\;we find that
$$S\ll_s\frac{\eta x}{\log x}\frac{(\log y)^{2v-2}}{v^{L(y)}}.$$ Setting
$v=L(y)/2\log\log y$ we deduce that \be\label{pfe10}S\ll_s\frac{\eta
x}{\log
x}\exp\Bigl\{-\frac{\psi(y)^2}4+O\Bigl(\frac{\psi(y)^3}{(\log\log
y)^{1/2}}\Bigr)\Bigr\}=o\Bigl(\frac{\eta x}{\log
x}\Bigr)\quad(y\to\infty).\ee Next, we turn to the estimation of
$S^\prime$. Note that for every $1/10\le v\le1$ we have that
\be\label{pfe12}\begin{split}S^\prime&\le
v^{-L(y)}\sum_{\substack{p+s\le x\\p\nmid s}}v^{\Omega(p+s;y)}
\sum_{\substack{[d_1,d_2]|p+s\\y<d_1<d_2\le z}}1\\
&=v^{-L(y)}\sum_{\substack{y<d_1<d_2\le
z\\(d_1d_2,s)=1}}v^{\Omega([d_1,d_2];y)}\sum_{\substack{p+s\le
x,\;p\nmid s\\p\equiv-s\;({\rm
mod}\;[d_1,d_2])}}v^{\Omega(\frac{p+s}{[d_1,d_2]};y)}.\end{split}\ee
Set $$S_1^\prime=\sum_{\substack{y<d_1<d_2\le
z\\(d_1d_2,s)=1\\(d_1,d_2)>y^2x^{-3/4}}}v^{\Omega([d_1,d_2];y)}\sum_{\substack{p+s\le
x,\;p\nmid s\\p\equiv-s\;({\rm
mod}\;[d_1,d_2])}}v^{\Omega(\frac{p+s}{[d_1,d_2]};y)}$$ and
$$S_2^\prime=\sum_{\substack{y<d_1<d_2\le
z\\(d_1d_2,s)=1\\(d_1,d_2)\le
y^2x^{-3/4}}}v^{\Omega([d_1,d_2];y)}\sum_{\substack{p+s\le
x,\;p\nmid s\\p\equiv-s\;({\rm
mod}\;[d_1,d_2])}}v^{\Omega(\frac{p+s}{[d_1,d_2]};y)}.$$ Put
$m=(d_1,d_2)$ and $d_i=mt_i$ so that $[d_1,d_2]=mt_1t_2$. Note that
$m\le z-y$. First, we deal with $S_1^\prime$. Since
$v^{\Omega(n;y)}\le v^{\omega(n;y)}$ for $v\le1$ and $[d_1,d_2]\le
2x^{3/4}$ in the range of $S_1^\prime$, Lemma \ref{l2}\;gives us
that
\be\label{pfe13}\begin{split}S_1^\prime&\ll_s\sum_{y^2x^{-3/4}<m\le
z-y}\quad\sum_{y/m<t_1<t_2\le z/m}\frac{v^{\Omega(mt_1t_2;y)}x(\log
y)^{v-1}}{\phi(mt_1t_2)\log x}\\
&\ll\frac{x(\log y)^{v-1}\log\log y}{\log x}\sum_{m\le
z-y}\frac{v^{\Omega(m)}}m\Bigl(\sum_{y/m<t\le
z/m}\frac{v^{\Omega(t)}}t\Bigr)^2,\end{split}\ee uniformly in
$1/10\le v\le1$, since $\Omega(n;y)\ge\Omega(n)-2$ for $n\le y^3$.
By relation (2.39) in \cite{hall_ten2}\;we have
\be\label{pfe16}\sum_{y/m<t\le
z/m}\frac{v^{\Omega(t)}}t\ll\eta\Bigl(\log\frac1\eta\Bigr)^{1-v}\Bigl(\log\frac
ym\Bigr)^{v-1}\asymp\eta(\log\log y)^{1-v}\Bigl(\log\frac
ym\Bigr)^{v-1},\ee which, combined with \eqref{pfe13}, yields that
\be\label{pfe14}S_1^\prime\ll_s\frac{\eta^2x(\log y)^{v-1}(\log\log
y)^{3-2v}}{\log x}\sum_{m\le
z-y}\frac{v^{\Omega(m)}}m\Bigl(\log\frac ym\Bigr)^{2v-2}.\ee We now
estimate $S_2^\prime$. First, for $d_1,d_2$ in the range of
summation of $S_2^\prime$ we have $x(d_1,d_2)/y^2\le x^{1/4}$, by
definition. So if $S_2^\prime$ is a non-empty sum, we must have that
$y\ge x^{3/8}$ and $m=(d_1,d_2)\le x^{1/4}\le y^{2/3}$.
Consequently,
$$S_2^\prime\le\sum_{\substack{m\le y^{2/3}\\(m,s)=1}}\sum_{\substack{y/m<t_1<t_2\le
z/m\\(t_1t_2,s)=1}}\sum_{\substack{p+s\le x,\;p\nmid
s\\p\equiv-s\;({\rm mod}\;mt_1t_2)}}v^{\Omega(p+s;y)}.$$ Set
$p+s=mt_1t_2k$. Then we have that $k\le x/(yt_1)$,
$\frac{z-y}m\ge(\frac zm)^{1/2}$ and $mt_1k\le(t_1kz)^{7/8}$. Also,
note that $\Omega(n;y)\ge\Omega(n)-2$ for $n\le x$, since $y\ge
x^{3/8}$. So
\be\begin{split}S_2^\prime&\le\frac1{v^2}\sum_{\substack{m\le
y^{2/3}\\(m,s)=1}}\sum_{\substack{y/m<t_1\le
z/m\\(t_1,s)=1}}\sum_{\substack{k\le
x/(yt_1)\\(k,s)=1}}v^{\Omega(mt_1k)}\sum_{\substack{t_1ky<p+s\le
t_1kz\\p\equiv-s\;({\rm
mod}\;mt_1k)}}v^{\omega(\frac{p+s}{mt_1k})}\nonumber\\
&\ll_s\sum_{m\le y^{2/3}}\sum_{y/m<t_1\le z/m}\sum_{k\le
x/(yt_1)}\frac{v^{\Omega(mt_1k)}t_1k(z-y)(\log(t_1kz))^{v-2}}{\phi(mt_1k)}\\
&\ll\frac{\eta y(\log y)^{v-1}\log\log y}{\log x}\sum_{m\le
y^{2/3}}\frac{v^{\Omega(m)}}m\sum_{y/m<t_1\le
z/m}v^{\Omega(t_1)}\sum_{k\le xm/y^2}v^{\Omega(k)}\\
&\ll\frac{\eta x(\log y)^{v-1}\log\log y}{y\log x}\sum_{m\le
y^{2/3}}v^{\Omega(m)}(\log2m)^{v-1}\sum_{y/m<t_1\le
z/m}v^{\Omega(t_1)},\end{split}\ee uniformly in $1/10\le v\le 1$, by
Lemma \ref{l2}\;and Theorem 01 in \cite{hall_ten2}. Also,
\be\begin{split}\sum_{y/m<t_1\le z/m}v^{\Omega(t_1)}\asymp\frac
ym\sum_{y/m<t_1\le z/m}\frac{v^{\Omega(t_1)}}{t_1}
&\ll\frac{\eta y(\log\log y)^{1-v}}m\Bigl(\log\frac ym\Bigr)^{v-1}\\
&\asymp\frac{\eta y(\log y)^{v-1}(\log\log
y)^{1-v}}m,\nonumber\end{split}\ee by \eqref{pfe16}, since $m\le
y^{2/3}$. Hence \be\label{pfe15}S_2^\prime\ll_s\frac{\eta^2x(\log
y)^{2v-2}(\log\log y)^{2-v}}{\log x}\sum_{m\le
y^{2/3}}\frac{v^{\Omega(m)}}m(\log2m)^{v-1}.\ee Inequalities
\eqref{pfe12}, \eqref{pfe14}\;and \eqref{pfe15}\;imply that
$$S^\prime\ll_s\frac{\eta^2xv^{-L(y)}(\log\log y)^{3-2v}}{\log
x}\sum_{m\le z-y}\frac{v^{\Omega(m)}}m(\log2m)^{v-1}\Bigl(\log\frac
ym\Bigr)^{2v-2}.$$ If we set $v=1/2$, by partial summation and the
estimate $\sum_{n\le x}v^{\Omega(n)}\ll x(\log2x)^{v-1}$ we find
that
$$\sum_{m\le z-y}\frac{v^{\Omega(m)}}m(\log m)^{v-1}\Bigl(\log\frac
ym\Bigr)^{2v-2}\ll\frac{\log\log y}{\log y}$$ and consequently
$$S^\prime\ll_s\frac{\eta^2x}{\log x}(\log
y)^{\log4-1}2^{\psi(y)\sqrt{\log\log y}}(\log\log y)^3.$$ Lastly,
putting $\psi(y)=\min\{\xi,(\log\log y)^{1/6}\}$ yields that
$$S^\prime\ll_s\frac{\eta x}{\log x}\frac{(\log\log y)^3}{e^{(1-\log2)\xi\sqrt{\log\log y}}}=o\Bigl(\frac{\eta x}{\log
x}\Bigr).$$ Inserting the above estimate and \eqref{pfe10}\;into
\eqref{pfe11}\;gives us that
$$\hs\ge f(s)\frac{315\zeta(3)}{2\pi^4}\frac{\eta x}{\log x}(1-o_s(1)),$$
which completes the proof of part (d) in the case that $z>y+(\log
y)^{-2}$ too.
\end{proof}


\section{Intermediate and large values of $\eta$}

To prove Theorem \ref{thminter1}\;we reduce the counting in $\hsd$
to the estimation of a sum involving $L(a;\eta)$ as done in
\cite{kf}\;for bounding $H(x,y,z)$; then we apply Lemma \ref{prl8}.
First, we show the following result. Theorem \ref{thminter1}\;will
then follow as an easy corollary.

\renewcommand{\labelenumi}{(\arabic{enumi})}

\begin{theorem}\label{imvl1} Fix $s\in\SZ\setminus\{0\}$ and $B\ge 2$. Let $x\ge x_0(s,B)$ and $3\le y+1\le z$ with $z\le
x^{2/3}$, $\eta\in[(\log y)^{-B},\frac{\log y}{100}]$ and $\{y<d\le
z:(d,s)=1\}\neq\emptyset\}$. Then for $$\frac x{(\log
x)^B}\le\Delta\le\frac x2$$ we have that
$$\hsd\gg_{s,B}\frac\Delta x\frac{H(x,y,z)}{\log x}.$$
\end{theorem}

\begin{proof} Fix $\Delta\in(x(\log x)^{-B},x/2]$ and set $s_1=2/(s,2)$. Let $y_0=y_0(s,B)$ be a
large positive constant. If $y\le y_0$, then
\be\begin{split}\hsd&\ge\max_{\substack{y<d\le
z\\(d,s)=1}}\Bigl(\pi(x-s;d,-s)-\pi(x-\Delta-s;d,-s)\Bigr)\nonumber\\
&\gg_{y_0}\frac\Delta{\log x}\asymp_{y_0}\frac\Delta
x\frac{H(x,y,z)}{\log x},\end{split}\ee by the Prime Number Theorem
for arithmetic progressions \cite[p. 123]{dav}\;and our assumption
that $\{y<d\le z:(d,s)=1\}\neq\emptyset$. Suppose now that $y>y_0$.
Fix an integer $v=v(s)\ge3$ and set $w=z^{1/20v}$. We will choose
$v$ later; till then, all implied constants will be independent of
$v$. Consider integers $n=aqb_1b_2s_1\in(x-\Delta,x]$ with
\begin{enumerate}\item $a\le w$, $\mu^2(a)=1$ and $(a,2s)=1$;
\item $\log(y/q)\in\mathcal{L}(a;\eta)$, $P^-(q)>w$ and $(q,2s)=1$;
\item $b_1\in\mathscr{P}(w,z)$ and $\tau(b_1)\le v^2$;
\item $P^-(b_2)>z$;
\item $n-s$ is prime.
\end{enumerate} Condition (2) implies that there exists $d|a$ such that $y/d<q\le z/d$; in particular, we
have that $\tau(n,y,z)\ge 1$ and thus $n$ is counted by $\hsd$.
Also, $\Omega(q)\le\frac{\log z}{\log w}=20v$ and therefore
$$\tau(qb_1)\le2^{\Omega(q)}\tau(b_1)\le2^{20v}v^2.$$ Since
each $n$ has at most $\tau(qb_1)\le2^{20v}v^2$ representations of
this form, we find that
\be\label{imve12}\begin{split}2^{20v}v^2\bigl(\hsd\bigr)&\ge\sum_{\substack{a\le
w\\\mu^2(a)=1\\(a,2s)=1}}\sum_{\substack{\log(y/q)\in\mathcal{L}(a;\eta)\\P^-(q)>w\\(q,2s)=1}}\sum_{\substack{(x-\Delta)/aqs_1<b_1b_2\le
x/aqs_1\\b_1\in\mathscr{P}(w,z),\;P^-(b_2)>z\\\tau(b_1)\le v^2\\aqb_1b_2s_1-s\;\text{prime}}}1\\
&=:\sum_{\substack{a\le
w\\\mu^2(a)=1\\(a,2s)=1}}\sum_{\substack{\log(y/q)\in\mathcal{L}(a;\eta)\\P^-(q)>w\\(q,2s)=1}}B_0(a,q).\end{split}\ee
Given $a$ and $q$ as above, put
$$B(a,q)=\sum_{\substack{(x-\Delta)/aqs_1<b\le
x/aqs_1\\P^-(b)>w\\aqbs_1-s\;\text{prime}}}1\quad\text{and}\quad
R(a,q)=B(a,q)-B_0(a,q).$$ Given $b$ with $P^-(b)>w$, write
$b=b_1b_2$ with $b_1\in\mathscr{P}(w,z)$ and $P^-(b_2)>z$ and put
$F(b)=\tau(b_1)$. Then, for fixed $a$ and $q$ with $(aq,2s)=1$, we
have that
$$R(a,q)\le\frac1{v^2}\sum_{\substack{(x-\Delta)/aqs_1<b\le
x/aqs_1\\P^-(b)>w\\aqbs_1-s\;\text{prime}}}F(b)=\frac1{v^2}\sum_{\substack{x-\Delta<p+s\le
x\\p\equiv-s\;({\rm
mod}\;aqs_1)\\P^-(\frac{p+s}{aqs_1})>w}}F\Bigl(\frac{p+s}{aqs_1}\Bigr)\ll_s\frac1v\frac{\Delta}{\phi(aq)\log
x\log w},$$ by Lemma \ref{l2}. Inserting the above estimate into
\eqref{imve12}\;yields that
\be\label{imve13}\begin{split}2^{20v}v^2\bigl(H(x,y,z;P_s)-H(x&-\Delta,y,z;P_s)\bigr)\ge\sum_{\substack{a\le
w\\\mu^2(a)=1\\(a,2s)=1}}\sum_{\substack{\log(y/q)\in\mathcal{L}(a;\eta)\\P^-(q)>w\\(q,2s)=1}}B(a,q)\\
&-O_s\Bigl(\frac1v\frac \Delta{\log x\log w}\sum_{\substack{a\le
w\\\mu^2(a)=1\\(a,2s)=1}}\sum_{\substack{\log(y/q)\in\mathcal{L}(a;\eta)\\P^-(q)>w\\(q,2s)=1}}\frac1{\phi(aq)}\Bigr).\end{split}\ee
Next, we need to approximate the characteristic function of integers
$n$ with $P^-(n)>w$ with a `smoother' function, the reason being
that the error term $\pi(x;rq,a)-{\rm li}(x)/\phi(rq)$ in Lemma
\ref{prl5}\;is weighted with the smooth function 1 as $q$ runs
through $[1,Q]\cap\SN$. To do this we appeal to Lemma
\ref{fundle}(a) with $Z=w$, $D=z^{1/20}$ and $\kappa=2$. Then
\be\label{imve1}\begin{split}2^{20v}v^2\bigl(&\hsd\bigr)\\&\ge\sum_{\substack{a\le
w\\\mu^2(a)=1,(a,2s)=1}}\sum_{\substack{\log(y/q)\in\mathcal{L}(a;\eta)\\(q,2s)=1}}(\lambda^-*1)(q)B(a,q)-O_s(\mathscr{R}_1)\\
&\ge\sum_{\substack{a\le
w\\\mu^2(a)=1,(a,2s)=1}}\sum_{\substack{\log(y/q)\in\mathcal{L}(a;\eta)\\(q,2s)=1}}(\lambda^+*1)(q)B(a,q)-O_s(\mathscr{R}_1+\mathscr{R}_2),\end{split}\ee
where $$\mathscr{R}_1:=\frac1v\frac\Delta{\log x\log
w}\sum_{\substack{a\le
w\\\mu^2(a)=1,(a,2s)=1}}\sum_{\substack{\log(y/q)\in\mathcal{L}(a;\eta)\\(q,2s)=1}}\frac{(\lambda^+*1)(q)}{\phi(aq)}$$
and $$\mathscr{R}_2:=\sum_{\substack{a\le
w\\\mu^2(a)=1,(a,2s)=1}}\sum_{\substack{\log(y/q)\in\mathcal{L}(a;\eta)\\(q,2s)=1}}\bigl((\lambda^+*1)(q)-(\lambda^-*1)(q)\bigr)B(a,q).$$
We now bound $\mathscr{R}_2$ from above. For fixed $a$ and $q$ with
$(aq,2s)=1$ we have $$B(a,q)\ll_s\frac\Delta{\phi(aq)\log x\log
w},$$ by the arithmetic form of the large sieve \cite{MonV}\;or
Lemma \ref{l2}. Since $\lambda^+*1-\lambda^-*1$ is always
non-negative, we get that
\be\label{imve2}\mathscr{R}_2\ll\frac\Delta{\log x\log
w}\sum_{\substack{a\le
w\\\mu^2(a)=1,(a,2s)=1}}\sum_{\substack{\log(y/q)\in\mathcal{L}(a;\eta)\\(q,2s)=1}}\frac{(\lambda^+*1)(q)-(\lambda^-*1)(q)}{\phi(aq)}.\ee
Fix $a\le w$ with $(a,2s)=1$ and let $\{I_r\}_{r=1}^R$ be the
collection of the intervals $[\log d-\eta,\log d)$ with $d|a$. Then
for $I=[\log d-\eta,\log d)$ in this collection Lemmas
\ref{prl7}\;and \ref{fundle}(a) imply that
\be\label{imve4}\begin{split}\sum_{\substack{\log(y/q)\in
3I\\(q,2s)=1}}&\frac{(\lambda^+*1)(q)-(\lambda^-*1)(q)}{\phi(aq)}\\
=&\sum_{\substack{c\le
z^{1/20}\\(c,2s)=1}}(\lambda^+(c)-\lambda^-(c))\sum_{\substack{e^{-\eta}y/cd<m\le
e^{2\eta}y/cd\\(m,2s)=1}}\frac1{\phi(acm)}\\
=&\frac{315\zeta(3)}{2\pi^4}\frac{g(2as)\phi(2s)}{2|s|\phi(a)}\sum_{\substack{c\le
z^{1/20}\\(c,2s)=1}}\frac{\lambda^+(c)-\lambda^-(c)}{c}\frac{g(ac)}{g(a)}\frac{c\phi(a)}{\phi(ac)}\Bigl(3\eta+O_s(y^{-2/3})\Bigr)\\
\ll_s&\frac{\eta}{e^v\phi(a)}\prod_{\substack{p\le w\\p\nmid
2s,p|a}}\Bigl(1-\frac1p\Bigr)\prod_{\substack{p\le w\\p\nmid
2sa}}\Bigl(1-\frac{g(p)}{p-1}\Bigr)+\frac1{\phi(a)\sqrt{y}}\asymp_s\frac1{e^v}\frac{\eta}{\phi(a)\log
w},\end{split}\ee provided that $y_0$ is large enough, since
$g(p)p/(p-1)\le\min\{p-1,2\}$ for $p\ge3$, $g(p)=1+O(p^{-2})$ and
$g(a)\asymp1$. By Lemma \ref{prl9}, there exists a sub-collection
$\{I_{r_t}\}_{t=1}^T$ of mutually disjoint intervals so that
$$\bigcup_{t=1}^T3I_{r_t}\supset\mathcal{L}(a;\eta).$$ Consequently
\be\begin{split}\sum_{\substack{\log(y/q)\in\mathcal{L}(a;\eta)\\(q,2s)=1}}\frac{(\lambda^+*1)(q)-(\lambda^-*1)(q)}{\phi(aq)}
&\le\sum_{t=1}^T\sum_{\substack{\log(y/q)\in3I_{r_t}\\(q,2s)=1}}\frac{(\lambda^+*1)(q)-(\lambda^-*1)(q)}{\phi(aq)}\nonumber\\
&\ll_s\frac1{e^v}\frac{T\eta}{\phi(a)\log w}\\
&=\frac1{e^v}\frac1{\phi(a)\log w}{\rm
meas}\Bigl(\bigcup_{t=1}^TI_{r_t}\Bigr)\\
&\le\frac1{e^v}\frac{L(a;\eta)}{\phi(a)\log w},\end{split}\ee since
$\lambda^+*1-\lambda^-*1$ is always non-negative. By the above
inequality and~\eqref{imve2} we get that
\be\label{imve3}\mathscr{R}_2\ll_s\frac1{e^v}\frac \Delta{\log
x\log^2w}\sum_{\substack{a\le
w\\\mu^2(a)=1,(a,2s)=1}}\frac{L(a;\eta)}{\phi(a)}.\ee We now bound
from the below the sum
$$\mathscr{S}:=\sum_{\substack{a\le
w\\\mu^2(a)=1,(a,2s)=1}}\sum_{\substack{\log(y/q)\in\mathcal{L}(a;\eta)\\(q,2s)=1}}(\lambda^+*1)(q)B(a,q).$$
We fix $a$ and $q$ with $(aq,2s)=1$ and seek a lower bound on
$B(a,q)$. By Lemma \ref{fundle}(b) applied with $Z=w$ and $D=w^3$,
there exists a sequence $\{\rho(d)\}_{d\le w^3}$ such that $\rho*1$
is bounded above by the characteristic function of integers $b$ with
$P^-(b)>w$. So, if we put
$$E(x;k,a)=\pi(x-s;k,a)-\pi(x-\Delta-s;k,a)-\frac{{\rm
li}(x-s)-{\rm li}(x-\Delta-s)}{\phi(k)},$$ then Lemma
\ref{fundle}(b) and the fact that $2|s_1s$ imply that
\be\begin{split}B(a,q)&=\sum_{\substack{x-\Delta<p+s\le
x\\p\equiv-s\;({\rm
mod}\;aqs_1)\\P^-((p+s)/aqs_1)>w}}1\ge\sum_{\substack{x-\Delta<p+s\le
x\\p\equiv-s\;({\rm
mod}\;aqs_1)\\p\nmid s}}(\rho*1)\Bigl(\frac{p+s}{aqs_1}\Bigr)\nonumber\\
&=\sum_{\substack{m\le w^3\\(m,s)=1}}\rho(m)\bigl(\pi(x-s;aqs_1m,-s)-\pi(x-s-\Delta;aqs_1m,-s)\bigr)+O_s(1)\\
&=\bigl({\rm li}(x-s)-{\rm li}(x-s-\Delta)\bigr)\sum_{\substack{m\le
w^3\\(m,s)=1}}\frac{\rho(m)}{\phi(aqs_1m)}+O_s(1)+\mathscr{R}_{aqs_1}^\prime\\
&\ge C_s\frac\Delta{\phi(aq)\log x\log
w}+\mathscr{R}_{aqs_1}^\prime\end{split}\ee for some positive
constant $C_s$, where
$$\mathscr{R}_{aqs_1}^\prime=\sum_{\substack{m\le
w^3\\(m,s)=1}}\rho(m)E(x;aqs_1m,-s).$$ Since $\lambda^+*1$ is always
non-negative, we deduce that \be\label{imv9}\mathscr{S}\ge
C_s\frac\Delta{\log x\log w}\sum_{\substack{a\le
w\\\mu^2(a)=1,(a,2s)=1}}\sum_{\substack{\log(y/q)\in\mathcal{L}(a;\eta)\\(q,2s)=1}}\frac{(\lambda^+*1)(q)}{\phi(aq)}
+\mathscr{R}^\prime,\ee where
$$\mathscr{R}^\prime=\sum_{\substack{a\le w\\\mu^2(a)=1,(a,2s)=1}}
\sum_{\substack{\log(y/q)\in\mathcal{L}(a;\eta)\\(q,2s)=1}}(\lambda^+*1)(q)\mathscr{R}^\prime_{aqs_1}.$$
Combining \eqref{imve1}, \eqref{imve3}\;and \eqref{imv9}\;we get
that
\be\label{imve14}\begin{split}2^{20v}v^2\bigl(&\hsd\bigr)\\
&\ge\frac{C_s}2\frac{\Delta}{\log x\log w}\sum_{\substack{a\le
w\\\mu^2(a)=1,(a,2s)=1}}\sum_{\substack{\log(y/q)\in\mathcal{L}(a;\eta)\\(q,2s)=1}}\frac{(\lambda^+*1)(q)}{\phi(aq)}\\
&\quad-O_s\Bigl(|\mathscr{R}^\prime|+\frac1{e^v}\frac \Delta{\log
x\log^2w}\sum_{\substack{a\le
w\\\mu^2(a)=1,(a,2s)=1}}\frac{L(a;\eta)}{\phi(a)}\Bigr),\end{split}\ee
provided that $v$ is large enough. Fix now $a\le w$ with $(a,2s)=1$
and look at the sum over $q$ on the right hand side of
\eqref{imve14}. Let $\{I_r\}_{r=1}^R$ be the collection of the
intervals $[\log d-\eta,\log d)$ with $d|a$. Then, using a similar
argument with the one leading to \eqref{imve4}, we find that for $I$
in this collection $$\sum_{\substack{\log(y/q)\in
I\\(q,2s)=1}}\frac{(\lambda^+*1)(q)}{\phi(aq)}\gg_s\frac{\eta}{\phi(a)\log
w},$$ provided that $y_0$ and $v$ are large enough. Moreover, by
Lemma \ref{prl9}, there exists a sub-collection
$\{I_{r_t}\}_{t=1}^T$ of mutually disjoint intervals so that $$\eta
T={\rm Vol}\Bigl(\bigcup_{t=1}^TI_{r_t}\Bigr)\ge\frac 13{\rm
Vol}\Bigl(\bigcup_{r=1}^RI_r\Bigr)=\frac{L(a;\eta)}3.$$ Hence
$$\sum_{\substack{\log(y/q)\in\mathcal{L}(a;\eta)\\(q,2s)=1}}\frac{(\lambda^+*1)(q)}{\phi(aq)}\ge\sum_{t=1}^T\sum_{\substack{\log(y/q)\in
I_{r_t}\\(q,2s)=1}}\frac{(\lambda^+*1)(q)}{\phi(aq)}\gg_s\frac{\eta
T}{\phi(a)\log w}\gg\frac{L(a;\eta)}{\phi(a)\log w},$$ where we used
the fact that $\lambda^+*1$ is non-negative. Inserting this
inequality into \eqref{imve14}\;and choosing a large enough $v$ we
conclude that \be\label{imve15}\hsd\ge M_s\frac\Delta{\log
x\log^2y}\sum_{\substack{a\le
w\\\mu^2(a)=1,(a,2s)=1}}\frac{L(a;\eta)}{\phi(a)}-O_s(|\mathscr{R}^\prime|)\ee
for some positive constant $M_s$. Furthermore, note that if $a$ is
squarefree, we may uniquely write $a=db$, where $d|2s$,
$\mu^2(d)=\mu^2(b)=1$ and $(b,2s)=1$, in which case
$L(a;\eta)\le\tau(d)L(b;\eta)$, by inequality \eqref{pre2}. Thus
$$\sum_{\substack{a\le
w\\\mu^2(a)=1}}\frac{L(a;\eta)}{\phi(a)}\le\sum_{d|2s,\mu^2(d)=1}\frac{\tau(d)}{\phi(d)}\sum_{\substack{b\le
w/d\\\mu^2(b)=1\\(b,2s)=1}}\frac{L(b;\eta)}{\phi(b)}\le\Bigl(\sum_{d|s}\frac{\tau(d)}{\phi(d)}\Bigr)\sum_{\substack{b\le
w\\\mu^2(b)=1\\(b,2s)=1}}\frac{L(b;\eta)}{\phi(b)},$$ which,
combined with \eqref{imve15}, Lemma \ref{prl8}\;and the trivial
inequality $\phi(a)\le a$, implies that $$\hsd\ge
M_s^\prime\frac\Delta x\frac{H(x,y,z)}{\log
x}-O_s(|\mathscr{R}^\prime|)$$ for some positive constant
$M_s^\prime$. In addition, observe that
$$H(x,y,z)\gg\frac x{(\log y)^B},$$ by Theorem \ref{thm1}\;and our assumption that
$\eta\ge(\log y)^{-B}$. Hence
$$\hsd\gg_s\frac\Delta x\frac{H(x,y,z)}{\log
x}\Bigl(1-O_s\Bigl(\frac{(\log x)(\log
y)^B|\mathscr{R}^\prime|}{\Delta}\Bigr)\Bigr).$$ So it suffices to
show that
$$|\mathscr{R}^\prime|\ll_s\frac\Delta{(\log x)(\log y)^{B+1}}.$$ For
every $a\in\SN$ there is a unique set $D_a$ of pairs $(d,d^\prime)$
with $d\le d^\prime$, $d|a$ and $d^\prime|a$ such that
$$\mathcal{L}(a;\eta)=\bigcup_{(d,d^\prime)\in D_a}[\log
d-\eta,\log d^\prime)$$ and the intervals $[\log d-\eta,\log
d^\prime)$ for $(d,d^\prime)\in D_a$ are mutually disjoint. With
this notation we have that
\be\begin{split}|\mathscr{R}^\prime|&=\Bigl\lvert\sum_a\sum_m\rho(m)\sum_{(d,d^\prime)\in
D_a}\sum_{\substack{y/d^\prime<q\le
z/d\\(q,2s)=1}}(\lambda^+*1)(q)E(x;ams_1q,-s)\Bigr\rvert\\
&=\Bigl\lvert\sum_a\sum_m\rho(m)\sum_{(d,d^\prime)\in
D_a}\sum_c\lambda^+(c)\sum_{\substack{y/cd^\prime<g\le
z/cd\\(g,2s)=1}}E(x;ams_1cg,-s)\Bigr\rvert\\
&\le\sum_{\substack{a\le w\\(a,2s)=1}}\sum_{\substack{m\le
w^3\\(m,s)=1}}\sum_{\substack{c\le
z^{1/20}\\(c,2s)=1}}\sum_{(d,d^\prime)\in
D_a}\Bigl\lvert\sum_{\substack{y/cd^\prime<g\le
z/cd\\(g,2s)=1}}E(x;ams_1cg,-s)\Bigr\rvert.\nonumber\end{split}\ee
So writing the inner sum as a difference of two sums we obtain that
\be\label{imv17}\begin{split}|\mathscr{R}^\prime|&\le2\sup_{y\le
t\le z}\biggl\{\sum_{\substack{a\le
w\\(a,2s)=1}}\sum_{\substack{m\le w^3\\(m,s)=1}}\sum_{\substack{c\le
z^{1/20}\\(c,2s)=1}}\sum_{f|ams_1c}\Bigl\lvert\sum_{\substack{g\le t/f\\(g,2s)=1}}E(x;ams_1cg,-s)\Bigr\rvert\biggr\}\\
&\le2\sup_{y\le t\le z}\biggl\{\sum_{\substack{r\le
2z^{7/60}\\(r,s)=1}}\tau_3(r)\sum_{f|r}\Bigl\lvert\sum_{\substack{g\le
t/f\\(g,2s)=1}}E(x;rg,-s)\Bigr\rvert\biggr\}\\
&\le4\sup_{y\le t\le z}\biggl\{\sum_{\substack{r\le
z^{1/8}\\(r,s)=1}}\tau_3(r)\sum_{f|r}\Bigl\lvert\sum_{\substack{g\le
t/f\\(g,s)=1}}E(x;rg,-s)\Bigr\rvert\biggr\},\end{split}\ee since
$w^4z^{1/20}\le z^{7/60}\le z^{1/8}/4$ for all $v\ge3$. Put
$\mu=1+(\log y)^{-B-7}$ and cover the interval $[1,z^{1/8}]$ by
intervals of the form $[\mu^n,\mu^{n+1})$ for $n=0,1,\dots,N$. We
may take $N\ll(\log y)^{B+8}$. Since
$|E(x;k,-s)|\ll\frac\Delta{\phi(k)\log x}$ for $k\le z^{9/8}\le
x^{3/4}$ with $(k,s)=1$ by Lemma \ref{prl4}, we have that
\be\begin{split}&\sum_{\substack{r\le
z^{1/8}\\(r,s)=1}}\tau_3(r)\sum_{n=0}^N\sum_{\substack{f|r\\\mu^n\le
f<\mu^{n+1}}}\Bigl\lvert\sum_{\substack{g\le
t/f\\(g,s)=1}}E(x;rg,-s)-\sum_{\substack{g\le
t/\mu^n\\(g,s)=1}}E(x;rg,-s)\Bigr\rvert\nonumber\\
&\ll\sum_{\substack{r\le
z^{1/8}\\(r,s)=1}}\tau_3(r)\sum_{n=0}^N\sum_{\substack{f|r\\\mu^n\le
f<\mu^{n+1}}}\sum_{t/\mu^{n+1}<g\le t/\mu^n}\frac\Delta{\phi(rg)\log x}\\
&\ll\frac{\Delta\log\mu}{\log x}\sum_{r\le
z^{1/8}}\frac{\tau_3(r)}{\phi(r)}\sum_{f|r}1\ll\frac\Delta{(\log
x)(\log y)^{B+1}}\end{split}\ee for all $t\in[y,z]$, by Lemma
\ref{prl7}, which is admissible. Combining the above estimate with
\eqref{imv17}\;we find that
\be\label{imv19}|\mathscr{R}^\prime|\ll\sup_{y\le t\le
z}\biggl\{\sum_{n=0}^N\sum_{\substack{r\le
z^{1/8}\\(r,s)=1}}\tau_3(r)\tau(r)\Bigl\lvert\sum_{\substack{g\le
t/\mu^n\\(g,s)=1}}E(x;rg,-s)\Bigr\rvert\biggr\}+\frac\Delta{(\log
x)(\log y)^{B+1}}.\ee Finally, since $$\frac x2\le x-\Delta\le
x\quad\text{and}\quad\Delta\ge\frac x{(\log x)^B},$$ Lemma
\ref{prl5}\;applied with $A=5B+56$ in combination with the
Cauchy-Schwarz inequality yields that
\be\begin{split}&\sum_{\substack{r\le
z^{1/8}\\(r,s)=1}}\tau_3(r)\tau(r)\Bigl\lvert\sum_{\substack{g\le
t/\mu^n\\(g,s)=1}}E(x;rg,-s)\Bigr\rvert\nonumber\\
&\ll\biggl(\frac{\Delta}{\log x}\sum_{r\le z^{1/8}}\sum_{g\le
t/\mu^n}\frac{(\tau_3(r)\tau(r))^2}{\phi(rg)}\biggr)^{1/2}\biggl(\sum_{\substack{r\le
z^{1/8}\\(r,s)=1}}\Bigl\lvert\sum_{\substack{g\le
t/\mu^n\\(g,s)=1}}E(x;rg,-s)\Bigr\rvert\biggr)^{1/2}\\
&\ll_s\sqrt{\Delta}(\log x)^{18}\frac{\sqrt{x}}{(\log
x)^{5B/2+28}}\le\frac\Delta{(\log x)^{2B+10}}\end{split}\ee for all
$t\in[y,z]$ and all $n\in\{0,1,\dots,N\}$, since $z^{1/8}\le
x^{1/12}$ and $z^{9/8}\le x^{3/4}$. Plugging this estimate into
\eqref{imv19}\;gives us that
$$|\mathscr{R}^\prime|\ll_sN\frac\Delta{(\log x)^{2B+10}}+\frac\Delta{(\log x)(\log
y)^{B+1}}\ll\frac\Delta{(\log x)(\log y)^{B+1}},$$ which is
admissible.
\end{proof}

We are now in position to complete the proof of Theorem
\ref{thminter1}.

\begin{proof}[Proof of Theorem \ref{thminter1}] Fix $\Delta\in(x(\log x)^{-B},x]$
and set $\Delta_1=\min\{\Delta,x/2\}$. If $\eta\le\frac{\log
y}{100}$, then the theorem follows immediately by Theorem
\ref{imvl1}\;and the trivial inequality
$$\hsd\ge\hs-H(x-\Delta_1,y,z;P_s).$$ On the other hand, if
$\eta\ge\frac{\log y}{100}$, then \be\begin{split}\hsd&\ge H(x,y,y^{101/100};P_s)-H(x-\Delta,y,y^{101/100};P_s)\nonumber\\
&\gg_s\frac\Delta x\frac{H(x,y,y^{101/100})}{\log
x}\asymp\frac\Delta x\frac{H(x,y,z)}{\log x},\end{split}\ee by
Theorem \ref{thm1}. In any case, Theorem \ref{thminter1}\;holds.
\end{proof}

Using Theorems \ref{thm2}\;and \ref{thminter1}\;together with the
fact that if $d|n$, then $(n/d)|d$ as well, we show Theorem
\ref{thminter2}.

\begin{proof}[Proof of Theorem \ref{thminter2}] We may assume that
$y>\sqrt{x}$; else the result follows from Theorems \ref{thm2}\;and
\ref{thminter1}\;with $\Delta=x$. For future reference, note the
trivial inequality
\be\label{imve7}\hs\ge\pi(z-s)-\pi(y-s)\asymp_{s,B}\frac{z-y}{\log
z}\ge\frac{\eta y}{\log x},\ee by the Prime Number Theorem. First,
suppose that $\eta\le\log^{-2}(5x/z).$ For $q\in\SN$ set
$$A_q=\{p+s\in(qy,qz]:p\equiv-s\pmod q\}.$$ If the shifted prime
$p+s\le x$ has a divisor $d\in(y,z]$, then writing $p+s=dq$ we have
that $q\le x/y$ and $p+s\in A_q$. So, by Lemma \ref{prl4}, we find
that \be\label{imve8}\begin{split}\hs\le\sum_{\substack{q\le
x/y\\(q,s)=1}}|A_q|+O_s(1)\ll_s\sum_{\substack{q\le
x/y\\(q,s)=1}}\frac{q(z-y)}{\phi(q)\log(z-y)}&\asymp_B\frac{\eta
y}{\log x}\sum_{\substack{q\le
x/y\\(q,s)=1}}\frac q{\phi(q)}\\
&\ll\frac{\eta x}{\log x}\asymp\frac{H(x,y,z)}{\log
x},\end{split}\ee by Theorem \ref{thm1}. This proves the upper bound
in Theorem \ref{thminter2}\;when $\eta\le\log^{-2}(5x/z)$. In order
to show the lower bound when $\eta\le\log^{-2}(5x/z)$ it suffices to
consider the case $z>x^{2/3}$, since if $z\le x^{2/3}$, then we
immediately obtain the result by Theorem \ref{imvl1}. If
$x/z\le2|s|+2$, then $y\asymp_sx$ and thus
$$\hs\gg_{s,B}\frac{\eta x}{\log x},$$ by \eqref{imve7}. Combining
this with Theorem \ref{thm1}\;we complete the proof in this case. So
assume that $x/z\ge2|s|+2$, in which case
$$\{x/2z<q\le x/z:(q,s)=1\}\neq\emptyset.$$ It is easy to see that
\be\label{imve9}\hs\ge\Bigl\lvert\bigcup_{\substack{x/2z<q\le
x/z\\(q,s)=1}}A_q\Bigr\rvert.\ee If we set
$$T(p)=\lvert\{x/2z<q\le x/z:(q,s)=1,\;p+s\in A_q\}\rvert,$$ then the Cauchy-Schwarz inequality and
\eqref{imve9}\;yield that \be\label{imve10a}\Bigl(\sum_{p+s\le
x}T(p)\Bigr)^2\le\hs\sum_{p+s\le x}T^2(p).\ee First, we estimate
$\sum_{p+s\le x}T(p)$. Let $C=C(B)>0$ be a constant so that
\be\label{imve7a}\sum_{\substack{q\le Q\\(q,s)=1}}\pi(X;q,-s)={\rm
li}(X)\sum_{\substack{q\le
Q\\(q,s)=1}}\frac1{\phi(q)}+O_{s,B}\Bigl(\frac X{(\log
X)^{B+2}}\Bigr)\ee for all $X\ge2$ and all $Q\le X(\log X)^{-C}$.
Such a constant exists by Lemma \ref{prl5}. If $x/z\le(\log
x)^{C+1}$, then the Siegel-Walfisz theorem \cite[p. 133]{dav}\;and
Lemma \ref{prl7}\;imply that
\be\label{imve10}\begin{split}\sum_{p+s\le
x}T(p)&=\sum_{\substack{x/2z<q\le
x/z\\(q,s)=1}}\bigl(\pi(qz-s;q,-s)-\pi(qy-s;q,-s)\bigr)\\
&\gg_{s,B}\sum_{\substack{x/2z<q\le
x/z\\(q,s)=1}}\frac{q(z-y)}{\phi(q)\log x}\asymp\frac{\eta x}{\log
x}\sum_{\substack{x/2z<q\le
x/z\\(q,s)=1}}\frac1{\phi(q)}\\
&\asymp_s\frac{\eta x}{\log x}.\end{split}\ee On the other hand, if
$x/z\ge(\log x)^{C+1}$, then \eqref{imve7a}\;and Lemma
\ref{prl7}\;yield that \be\label{imve10b}\begin{split}\sum_{p+s\le
x}T(p)&\ge\sum_{x/2<p+s\le 2x/3}\sum_{\substack{\frac{p+s}z\le
q<\frac{p+s}y\\(q,s)=1,q|(p+s)}}1\\
&=\sum_{\substack{y<d\le
z\\(d,s)=1}}\Bigl(\pi(2x/3-s;d,-s)-\pi(x/2-s;d,-s)\Bigr)+O_s(1)\\
&=\Bigl({\rm li}(2x/3-s)-{\rm li}(x/2-s)\Bigr)\sum_{\substack{y<d\le
z\\(d,s)=1}}\frac1{\phi(d)}+O_{s,B}\Bigl(\frac x{(\log x)^{B+2}}\Bigr)\\
&\asymp_{s,B}\frac{\eta x}{\log x},\end{split}\ee since
$\eta\le\log^{-2}5\le\log(3/2)$. Also,
\be\label{imve10c}\sum_{p+s\le x}T(p)=\sum_{\substack{x/2z<q\le
x/z\\(q,s)=1}}|A_q|\ll_{s,B}\frac{\eta x}{\log x},\ee by
\eqref{imve8}. Combining inequalities \eqref{imve10},
\eqref{imve10b}\;and \eqref{imve10c}\;we deduce that
\be\label{imve10d}\sum_{p+s\le x}T(p)\asymp_{s,B}\frac{\eta x}{\log
x},\ee uniformly in $\eta\le\log^{-2}(5x/z)$ and $x/z\ge2|s|+2$. We
now bound from above the sum \be\label{imve11a}\sum_{p+s\le
x}T^2(p)=\sum_{p+s\le x}T(p)+\sum_{p+s\le x}T(p)(T(p)-1).\ee We have
that \be\label{imve11}\begin{split}\sum_{p+s\le
x}T(p)(T(p)-1)&=\sum_p\sum_{\substack{x/2z<q_1\le
x/z\\\frac{p+s}z\le
q_1<\frac{p+s}y\\q_1|(p+s),(q_1,s)=1}}\sum_{\substack{x/2z<q_2\le
x/z\\\frac{p+s}z\le q_2<\frac{p+s}y\\q_2|(p+s),\;(q_2,s)=1\\q_2\neq
q_1}}1\\
&=2\sum_{\substack{\frac x{2z}<q_1<q_2\le\frac
xz\\(q_1q_2,s)=1}}\sum_{\substack{p\equiv-s\;({\rm
mod}\;[q_1,q_2])\\q_2y<p+s\le q_1z}}1.\end{split}\ee Note that we
must have $q_2<e^\eta q_1$; otherwise, the corresponding summand on
the right hand side of \eqref{imve11}\;is trivially zero. Lemma
\ref{prl4}\;and the trivial estimate $\pi(x+h;q,a)-\pi(x;q,a)\le
h/q+1$ imply \be\label{imve11b}\sum_{\substack{p\equiv-s\;({\rm
mod}\;[q_1,q_2])\\q_2y<p+s\le
q_1z}}1\ll_s\frac{q_1z-q_2y}{\phi([q_1,q_2])\log\bigl(3+(q_1z-q_2y)/[q_1,q_2]\bigr)}+1.\ee
Set $m=(q_1,q_2)$ and $q_i=mt_i$, $i=1,2$, in the right hand side of
\eqref{imve11}. Then we will have $m\le x/2z$ and $t_1<t_2<e^\eta
t_1$. With this notation \eqref{imve11}\;and \eqref{imve11b}\;yield
that \be\label{imve11c}\begin{split}\sum_{p+s\le
x}T(p)(T(p)-1)&\ll_s\log\log(x/y)\sum_{m\le\frac x{2z}}\sum_{\frac
x{2mz}<t_1\le\frac x{mz}}\sum_{t_1<t_2<e^\eta
t_1}\frac{z/t_2-y/t_1}{\log(3+z/t_2-y/t_1)}\\
&\quad+\frac xz\log(x/z)+\eta\Bigl(\frac xz\Bigr)^2\end{split}\ee
Fix $m$ and $t_1$. Recall that we have assumed that $z>x^{2/3}$ and
$(\log x)^{-B}\ll\eta\le(\log(5x/z))^{-2}$. So
$\log\frac{z-y}{t_1}\gg_B\log x$ and consequently
\be\begin{split}\sum_{t_1<t_2<e^\eta
t_1}\frac{z/t_2-y/t_1}{\log(3+z/t_2-y/t_1)}&\le\int_{t_1}^{e^\eta
t_1}\frac{z/u-y/t_1}{\log(3+z/u-y/t_1)}du\\
&=\int_0^{(z-y)/t_1}\frac w{\log(w+3)}\frac z{(w+y/t_1)^2}dw\\
&\asymp_B\frac{\eta^2y}{\log x},\nonumber\end{split}\ee which,
combined with \eqref{imve10d}, \eqref{imve11a}\;and \eqref{imve11c},
yields that
$$\sum_{p+s\le x}T^2(p)\ll_{s,B}\frac{\eta x}{\log x}+\frac{\eta^2x}{\log
x}\log(x/y)\log\log(x/y)\ll\frac{\eta x}{\log x}.$$ Plugging the
above estimate and \eqref{imve10d}\;into \eqref{imve10a}\;gives us
that
$$\hs\gg_{s,B}\frac{\eta x}{\log x}\asymp\frac{H(x,y,z)}{\log x},$$
by Theorem \ref{thm1}. This completes the proof of the theorem in
the case when $\eta\le\log^{-2}(5x/z)$. Assume now that
$\eta\ge\log^{-2}(5x/z)$. Fix a large positive constant
$y_0=y_0(s,B)$. If $x/z\le y_0$, then $\eta\ge\log^{-2}(5y_0)$.
Hence \eqref{imve7}\;implies that
$$\hs\gg_{s,B}\frac{z-y}{\log y}\gg_{y_0}\frac z{\log
y}\gg_{y_0}\frac x{\log x}.$$ Combining the above inequality with
the trivial estimate $\hs\le\pi(x-s)$ and Theorem \ref{thm1}\;we
deduce that
$$\hs\asymp_{y_0}\frac x{\log x}\asymp_{y_0}\frac{H(x,y,z)}{\log x},$$ which shows the
desired result in this case. So suppose that $x/z>y_0$. We proceed
as in the proof of Theorem 1 (iv) in \cite{kf}. Partition $(\frac
x{\log^2(x/z)},x]$ into intervals $(x_1,x_2]$ with
$$\frac{x_2}{\log^3(x/z)}\le x_2-x_1\le\frac{2x_2}{\log^3(x/z)}.$$
Observe that if $p+s\in(x_1,x_2]$, then
$$\tau\Bigl(p+s,\frac{x_2}z,\frac{x_1}y\Bigr)\ge1\Rightarrow\tau(p+s,y,z)\ge1\Rightarrow\tau\Bigl(p+s,\frac{x_1}z,\frac{x_2}y\Bigr)\ge1.$$
So
\be\label{imve20}\begin{split}\hs\le\sum_{x_1,x_2}\Bigl\{H\Bigl(x_2,\frac{x_1}z,\frac{x_2}y;P_s\Bigr)&-H\Bigl(x_1,\frac{x_1}z,\frac{x_2}y;P_s\Bigr)\Bigr\}
+O_s\Bigl(\frac x{\log x\log^2(x/z)}\Bigr).\end{split}\ee Fix such
an interval $(x_1,x_2]$. Then
$$\log\Bigl(\frac{x_1}z\Bigr)\asymp\log\Bigl(\frac xz\Bigr),\quad x_2-x_1\ge\frac{x_2}{\log^4(x_2/y)},
\quad\log\Bigl(\frac{x_2/y}{x_1/z}\Bigr)\asymp\eta,\quad\frac{x_1}z\le\sqrt{x_2},$$
provided that $y_0$ is large enough. Therefore Theorems
\ref{thm1}\;and \ref{thm2}\;and Lemma \ref{cor1}\;imply that
\be\label{imve21}\begin{split}H\Bigl(x_2,\frac{x_1}z,\frac{x_2}y;P_s\Bigr)-H\Bigl(x_1,\frac{x_1}z,\frac{x_2}y;P_s\Bigr)\ll_{s,B}\frac{x_2-x_1}{x_2\log
x_2}H\Bigl(x_2,\frac{x_1}z,\frac{x_2}y\Bigr)&\asymp\frac{x_2-x_1}{x\log x}H\Bigl(x,\frac xz,\frac xy\Bigr)\\
&\asymp\frac{x_2-x_1}{x\log x}H(x,y,z).\nonumber\end{split}\ee
Inserting the above inqequality into \eqref{imve20}\;and summing
over $x_1,x_2$ completes the proof of the desired upper bound. The
corresponding lower bound is obtained in a similar fashion starting
from
$$\hs\ge\sum_{x_1,x_2}\Bigl\{H\Bigl(x_2,\frac{x_2}z,\frac{x_1}y;P_s\Bigr)-H\Bigl(x_1,\frac{x_2}z,\frac{x_1}y;P_s\Bigr)\Bigr\}$$
and using Theorem \ref{thminter1}\;in place of Theorem \ref{thm2}.
\end{proof}

We conclude this section with the proof of Theorem \ref{thmlarge}.

\begin{proof}[Proof of Theorem \ref{thmlarge}] Let $2\le y\le z\le x$. Let
$P=\prod_{y<p\le z}p$ be the product of all prime numbers in
$(y,z]$. Then \be\label{lv1}0\le\pi(x-s)-\hs\le\lvert\{p\le
x-s:(p+s,P)=1\}\rvert.\ee Lemma \ref{l2}\;implies that the right
hand side of \eqref{lv1}\;is $$\ll_s\frac x{\log x}\frac{\log
y}{\log z},$$ which combined with the Prime Number Theorem completes
the proof.
\end{proof}

\medskip

{\bf Acknowledgements.} The author would like to thank Kevin Ford
for many valuable suggestions.

\bigskip\bigskip

\bibliographystyle{amsplain}

\end{document}